\documentclass[11pt]{article}

\usepackage[T1]{fontenc}
\usepackage[utf8]{inputenc}
\usepackage[brazil]{babel}
\usepackage[nottoc,notlot,notlof]{tocbibind}
\usepackage{enumerate}
\usepackage{multirow,booktabs}
\usepackage[table]{xcolor}
\usepackage{fullpage}
\usepackage{lastpage}
\usepackage{indentfirst}
\usepackage{fancyhdr}
\usepackage{mathrsfs}
\usepackage{wrapfig}
\usepackage{setspace}
\usepackage{calc}
\usepackage{multicol}
\usepackage{cancel}
\usepackage[retainorgcmds]{IEEEtrantools}
\usepackage[margin=3cm]{geometry}

\usepackage{empheq}
\usepackage{framed}
\usepackage[most]{tcolorbox}
\usepackage{xcolor}
\colorlet{shadecolor}{orange!15}
\usepackage{footnote}

\usepackage{amsmath,amsfonts,amssymb,amscd}

\usepackage{bm, bbm}

\usepackage[all,2cell]{xy} \UseAllTwocells \SilentMatrices

\usepackage{hyperref}

\usepackage{subfiles}

\usepackage{showlabels}

\usepackage{lineno}

\usepackage{amsthm}
\newtheorem{teo}{Theorem}[section]
\newtheorem{lem}[teo]{Lemma} 
\newtheorem{cor}[teo]{Corollary}
\newtheorem{prop}[teo]{Proposition}

\newtheorem{defn}[teo]{Definition}
\newtheorem{ex}[teo]{Example}

\newtheorem{fat}[teo]{Fact}
\newtheorem*{claim*}{Claim}

\usepackage{filecontents}

\begin{document}

\title{K-theories and Free Inductive Graded Rings in Abstract Quadratic Forms Theories}

\author{
Kaique Matias de Andrade Roberto\footnote{Instituto de Matemática e Estatística, Universidade de São Paulo, Brazil.} \\ Hugo Luiz Mariano\footnote{Emails: kaique.roberto@alumni.usp.br, hugomar@ime.usp.br}
}

\maketitle

\begin{abstract}
 We build on previous work on multirings (\cite{roberto2021quadratic}) that provides generalizations of the available abstract quadratic forms theories (special groups and real semigroups) to the context of multirings (\cite{marshall2006real}, \cite{ribeiro2016functorial}). Here we raise one step in this generalization, introducing the concept of pre-special hyperfields and expand a fundamental tool in quadratic forms theory to the more general multivalued setting: the K-theory. We introduce and develop the K-theory of hyperbolic hyperfields that generalize simultaneously Milnor's K-theory (\cite{milnor1970algebraick}) and Special Groups K-theory, developed by Dickmann-Miraglia (\cite{dickmann2006algebraic}). We develop some properties of this generalized K-theory, that can be seen as a free inductive graded ring, a concept introduced in \cite{dickmann1998quadratic} in order to provide a solution of Marshall's Signature Conjecture.

 \textbf{Keywords:} quadratic forms, special groups, K-theory, multirings, hyperfields.
\end{abstract}

\section{Introduction}

Concerning Abstract Theories of Quadratic forms (in particular special groups and real semigroups), the references \cite{dickmann2000special}, \cite{dickmann2015faithfully} and \cite{dickmann2004real} are central. The theory of special groups deals simultaneously reduced and non-reduced theories but focuses on rings with an ``expressive amount'' of invertible coefficients to quadratic forms and the theory of real semigroups consider general coefficients of a ring, but only addresses the reduced case. Both are first-order theory, thus they allow the use of model theoretic methods.

M. Marshall in \cite{marshall2006real} introduced an approach to (reduced) theory of quadratic forms through the concept of multiring (roughly, a ring with a multi valued sum)
\footnote{The main terminology in the literature is "hyperring". Moreover, M. Marshall makes a distinction between "multiring" and "hyperrings" which is important in the context of quadratic forms. But throughout this entire work, we deal essentially with multifields/hyperfields and then, the main terminology here will be "hyperfield".}: this seems more intuitive for an algebraist since it encompasses (generalizes, in fact) some techniques of ordinary Commutative Algebra.  Moreover, the multirings encode copies of special groups and real semigroups (see \cite{ribeiro2016functorial}) and still allows the use of model-theoretic tools,  since multirings (hyperrings) endowed with convenient notion of morphisms constitutes a category that is isomorphic to a category of appropriate first-order structures.

In the recent work \cite{roberto2021quadratic}: (i) we have considered interesting pairs $(A,T)$ where $A$ is  a multiring and $T \subseteq A$ is a certain multiplicative subset in such a way to obtain models of abstract theories of quadratic forms (special groups and real semigroups) via natural quotients - Marshall's quotient construction and (ii) we have used this new setting to motivate a "non reduced" expansion of the theory  of real semigroups to deal the formally real case, isolating axioms over pairs involving multirings and a  multiplicative subset with some properties.

The uses of K-theoretic (and Boolean) methods in abstract theories of quadratic forms has been proved a very successful method, see for
instance, these two papers of Dickmann and Miraglia: \cite{dickmann1998quadratic} where they give an affirmative answer to Marshall's Conjecture, and \cite{dickmann2003lam}, where they give an affirmative answer to Lam's Conjecture (previously both conjecture have kept open for almost three decades).  These two central papers makes us take a deeper look at the theory of special groups (and hence, hyperbolic/pre-special hyperfields) by itself. This is not mere exercise in abstraction: from Marshall's and Lam's Conjecture many questions arise in the abstract and concrete context of quadratic forms.

In the present paper we provide  some new steps towards the  development of tools of algebraic theory of quadratic forms in this multiring setting: we have defined and explored K-theory and graded rings  in the context of hyperfields that, in particular, provides a generalization and unification of Milnor's K-theory (\cite{milnor1970algebraick}) and  special groups K-theory (\cite{dickmann2006algebraic}). We develop some properties of this generalized K-theory, that can be seen as a free inductive graded ring. This is a preprint version of \cite{roberto2021ktheory}.

{\bf Outline:}
In Section 2  we: recall the basic notions of K-theory in quadratic forms theories (Milnor's K-theory, Special Groups K-theory);  provide the basic definitions, constructions  and results on multirings  and hyperfields;  introduce the concept of pre-special  hyperfield.
In the third
section we introduce the notion of  K-theory of hyperfields and establish its main properties and some technical results.
In Section 4 we recognize the K-theory of special hyperfields as a   free construction in the category of inductive graded rings (\cite{dickmann2000special}) and prove that this concept, encompasses both Milnor's K-theory and  Special Groups K-theory.
We finish the work indicating some themes of research motivated by the present paper.

Throughout this entire paper, we adopt the following convention: let $A$ be an arbitrary structure (field, graded ring and so on) and let
$T\subseteq A$ be a subset such that the quotient $A/T$ is defined. We usually denote the elements of $A/T$ by $\overline a\in A/T$. This
notation is overwhelming here (as you will see rapidly). So, we will omit the overline symbol, and an equation ``$\overline a=\overline b$'' that took place in $A/T$ will be denoted simply by ``$a=b$ in $A/T$''. In other words, we will add ``in $A/T$'' or ``in $A$'' to make distinctions between some set and its quotients.

\section{Preliminaries}

This section contains, basically,  the fundamental definitions and results included for the convenience of the reader such as multirings, hyperfields, special groups, Minor's K-theory of fields and Dickmann-Miraglia K-theory of (pre) special groups. For more details, consult \cite{milnor1970algebraick}, \cite{dickmann2000special}, \cite{dickmann2006algebraic}, \cite{marshall2006real}, \cite{ribeiro2016functorial} or \cite{roberto2021quadratic}.

 \subsection{Milnor's K-theory}
  Here we get some definitions and results about Milnor's K-theory, as developed in \cite{milnor1970algebraick}.

  \begin{defn}[The Milnor's K-theory of a Field \cite{milnor1970algebraick}]\label{milkt}
  For a field $F$ (of characteristic not 2), $K_*F$ is the graded ring
$$K_*F=(K_0F,K_1F,K_2F,...)$$
  defined by the following rules: $K_0F:=\mathbb Z$. $K_1F$ is the multiplicative group $\dot F$ written additively.
  With this purpose, we fix the canonical ``logarithm'' isomorphism
$$l:\dot F\rightarrow K_1F,$$
  where $l(ab)=l(a)+l(b)$. Then $K_nF$ is defined to be the quotient of the tensor algebra
  $$K_1F\otimes K_1F\otimes...\otimes K_1F\,(n \mbox{ times})$$
  by the (homogeneous) ideal generated by all $l(a)\otimes l(1-a)$, with $a\ne0,1$. We also have the reduced K-theory graded ring
  $k_*F=(k_0F,k_1F,...,k_nF,...),$
  which is defined by the rule $k_nF:=K_nF/2K_nF$ for all $n\ge0$.
  \end{defn}

\begin{teo}[Theorem 4.1 of \cite{milnor1970algebraick}]
 There is only one morphism
 $$s_n:k_nF\rightarrow I^nF/I^{n+1}F$$
 which carries each product $l(a_1)...l(a_n)$ in $K_nF/2K_nF$ to the product
 $(\langle a_1\rangle-\langle1\rangle)...(\langle a_n\rangle-\langle1\rangle)$ modulo $I^{n+1}F$.

These morphisms determines an epimorphism $s_*:k_*F\rightarrow W_g(F)$, where
 $$W_g(F)=(WF/IF,IF/I^2F,...,I^nF/I^{n+1}F,...).$$
\end{teo}

For a field $F$, let $F_s$ be the a separable closure of $F$ and $G_F=\mbox{Gal}(F_s)$. Then, the exact sequence
$$\xymatrix{1\ar[r] &\{\pm1\}\ar[r] &\dot F_s\ar[r]^2 &\dot F_s\ar[r] & 1}$$
is taken to the following exact sequence
$$\xymatrix{H^0(G_F,\dot F_s)\ar[r]^2 & H^0(G_F,\dot F_s)\ar[r]^{\delta} &H^1(G_F,\{\pm1\})\ar[r] &H^1(G_F,\dot F_s)}$$
of cohomology groups. Identifying the two first groups with $\dot F$, and $\{\pm1\}$ with $\mathbb Z/2\mathbb Z$ and applying Hilbert's 90, we
have
$$\xymatrix{\dot F\ar[r]^2 & \dot F\ar[rr]^{\delta} & & H^1(G_F,\mathbb Z/2\mathbb Z)\ar[r] & 0}.$$
The quotient $\dot F/\dot F^2$ is identified with $k_1F$.

\begin{teo}[Lemma 6.1 of \cite{milnor1970algebraick}]
 The isomorfism $l(a)\mapsto\delta(a)$ from $K_1F/2K_1F$ to $H^1(G_F,\mathbb Z/2\mathbb Z)$ admits a unique extension to a graded ring morphism
 $$h_f:k_*F\rightarrow H^*(G_F,\mathbb Z/2\mathbb Z).$$
\end{teo}
The \textit{Milnor Conjecture} consists of the assertion that $s$ and $h$ are graded rings isomorfisms, which makes the fuctors
$K_*F/2K_*F,W_g(F),H^*(G,\mathbb Z/2\mathbb Z)$ isomorphic.

\subsection{Dickmann-Miraglia K-theory for Special Groups}

  There are some generalizations of Milnor's K-theory. In the quadratic forms context, maybe the most significant one is the Dickmann-Miraglia K-theory of Special Groups. It is a main tool in the proof of Marshall's and Lam's Conjecture. In this section, we get some definitions and results from \cite{dickmann2000special} and \cite{dickmann2006algebraic}.

  Firstly, we make a brief summary on special groups. Let $A$ be a set and $\equiv$ a binary relation on $A\times A$. We extend $\equiv$ to a binary relation $\equiv_n$ on $A^n$, by induction on $n\ge1$, as follows:
\begin{enumerate}[i -]
\item $\equiv_1$ is the diagonal relation $\Delta_A \subseteq A \times A$.

 \item $\equiv_2=\equiv$.
 \item If $n \geq 3$, $\langle a_1,...,a_n\rangle\equiv_n\langle b_1,...,b_n\rangle$ if and only there are  $x,y,z_3,...,z_n\in A$ such that
 $$\langle a_1,x\rangle\equiv\langle b_1,y\rangle,\,
 \langle a_2,...,a_n\rangle\equiv_{n-1}\langle x,z_3,...,z_n\rangle\mbox{ and }\langle b_2,...,b_n\rangle\equiv_{n-1}\langle y,z_3,...,z_n\rangle.$$
\end{enumerate}

Whenever clear from the context, we frequently abuse notation and indicate the afore-described extension $\equiv$ by the same symbol.

\begin{defn}[Special Group, 1.2 of \cite{dickmann2000special}]\label{defn:sg}
 A \textbf{special group} is a tuple $(G,-1,\equiv)$, where $G$ is a group of exponent 2,
i.e, $g^2=1$ for all $g\in G$; $-1$ is a distinguished element of $G$, and $\equiv\subseteq G\times
G\times G\times G$ is a relation (the special relation), satisfying the following axioms for all
$a,b,c,d,x\in G$:
\begin{description}
 \item [SG 0] $\equiv$ is an equivalence relation on $G^2$;
 \item [SG 1] $\langle a,b\rangle\equiv \langle b,a\rangle$;
 \item [SG 2] $\langle a,-a\rangle\equiv\langle1,-1\rangle$;
 \item [SG 3] $\langle a,b\rangle\equiv\langle c,d\rangle\Rightarrow ab=cd$;
 \item [SG 4] $\langle a,b\rangle\equiv\langle c,d\rangle\Rightarrow\langle
a,-c\rangle\equiv\langle-b,d\rangle$;
 \item [SG 5]
$\langle a,b\rangle\equiv\langle c,d\rangle\Rightarrow\langle ga,gb\rangle\equiv\langle
gc,gd\rangle,\,\mbox{for all }g\in G$.
 \item [SG 6 (3-transitivity)] the extension of $\equiv$ for a binary relation on $G^3$ is a
transitive relation.
\end{description}
\end{defn}

A group of exponent 2, with a distinguished element $-1$, satisfying the axioms SG0-SG3 and SG5 is called a {\bf  proto special group}; a \textbf{pre special group} is a proto special group that also satisfy SG4. Thus a \textbf{special group} is a pre-special group that satisfies SG6 (or, equivalently, for each $n \geq 1$, $\equiv_n$ is an equivalence relation on $G^n$).

A \textbf{$n$-form} (or form of dimension $n\ge1$) is an $n$-tuple of elements of a pre-SG $G$. An element $b\in G$ is \textbf{represented} on $G$ by the form $\varphi=\langle a_1,...,a_n\rangle$, in symbols $b\in D_G(\varphi)$, if there exists $b_2,...,b_n\in G$ such that $\langle b,b_2,...,b_n\rangle\equiv\varphi$.

A pre-special group (or special group)
$(G,-1,\equiv)$ is:\\
$\bullet$ \ \textbf{formally real} if $-1 \notin \bigcup_{n \in \mathbb{N}} D_G( n\langle 1 \rangle)$\footnote{Here the notation $n\langle 1 \rangle$ means the form $\langle a_1,...,a_n\rangle$ where $a_j=1$ for all $j=1,...,n$. In other words, $n\langle 1 \rangle$ is the form $\langle 1 ,...,1\rangle$ with $n$ entries equal to 1.} ;\\
$\bullet$ \ \textbf{reduced} if it is formally real and, for each $a \in G$, $a \in D_G(\langle 1, 1 \rangle)$ iff $a =1$.

\begin{defn}[1.1 of \cite{dickmann2000special}]\label{defnmorph}
 A map $\xymatrix{(G,\equiv_G,-1)\ar[r]^f & (H,\equiv_H,-1)}$ between pre-special groups is a \textbf{morphism
of pre-special groups or PSG-morphism} if $f:G\rightarrow H$ is a homomorphism of groups, $f(-1)=-1$ and for all
$a,b,c,d\in G$
$$\langle a,b\rangle\equiv_G\langle c,d\rangle\Rightarrow
\langle f(a),f(b)\rangle\equiv_H\langle f(c),f(d)\rangle$$
A \textbf{morphism of special groups or SG-morphism} is a pSG-morphism between the corresponding pre-special groups. $f$
will be an isomorphism if is bijective and $f,f^{-1}$ are PSG-morphisms.
\end{defn}

It can be verified that a special group  $G$ is formally real iff it admits some SG-morphism $f : G \to 2$. The category of special groups (respectively reduced special groups) and their morphisms will be denoted by $\mathcal{SG}$
(respectively $\mathcal{RSG}$).

\begin{defn}[The Dickmann-Miraglia K-theory \cite{dickmann2006algebraic}]\label{defn:ksg}
 For each special group $G$ (written multiplicatively) we associate a graded ring
 $$k_*G=(k_0G,k_1G,...,k_nG,...)$$
 as follows: $k_0G:=\mathbb F_2$ and $k_1G:=G$ written additively. With this purpose, we fix the canonical ``logarithm'' isomorphism
 $\lambda:G\rightarrow k_1G$, $\lambda(ab)=\lambda(a)+\lambda(b)$. Observe that $\lambda(1)$ is the zero of $k_1G$ and $k_1G$ has exponent 2, i.e, $\lambda(a)=-\lambda(a)$ for all $a\in G$. In the sequel, we define $k_*G$ by the quotient of the $\mathbb F_2$-graded algebra
 $$(\mathbb F_2,k_1G,k_1G\otimes_{\mathbb F_2} k_1G,k_1G\otimes_{\mathbb F_2} k_1G\otimes_{\mathbb F_2} k_1G,...)$$
 by the (graded) ideal generated by $\{\lambda(a)\otimes\lambda(ab),\,a\in D_G(1,b)\}$. In other words, for each $n\ge2$,
$$k_nG:=T^n(k_1G)/Q^n(G),$$
where
$$T^n(k_1G):=k_1G\otimes_{\mathbb F_2} k_1G\otimes_{\mathbb F_2}...\otimes_{\mathbb F_2} k_1G$$
and $Q^n(G)$ is the subgroup generated by all expressions of type $\lambda(a_1)\otimes\lambda(a_2)\otimes...\otimes\lambda(a_n)$ such that for
some $i$ with $1\le i< n$, there exist $b\in G$ such that $a_i\in D_G(1,b)$ and $a_i=a_{i+1}b$, which in symbols, means
\begin{align*}
 Q^n(G)&:=\langle\{\lambda(a_1)\otimes\lambda(a_2)\otimes...\otimes\lambda(a_n):\mbox{ exists }1\le i< n\mbox{ and }\,b\in G \\
 &\mbox{such that }a_i=a_{i+1}b\mbox{ and }a_i\in D_G(1,b)\}\rangle.
\end{align*}
\end{defn}

Since $\lambda(a)+\lambda(a)=0$ for all $a\in k_1G$, it follow that $\eta+\eta=0$ for all $\eta\in k_nG$, so this is a group of exponent 2.

Moreover, we will denote ``$\lambda(a_1)\otimes\lambda(a_2)\otimes...\otimes\lambda(a_n)$'' simply by
``$\lambda(a_1)\lambda(a_2)...\lambda(a_n)$''. Clearly, $k_\ast(G)$ is a graded ring, and in particular, a ring, so that we are able to multiply
elements $\eta\in k_nG$, $\tau\in k_mG$. Whenever we want to do this, we will denote $\eta\cdot\tau$,
in order to avoid confusion with the simplifications described above.

Finally, since we only take tensorial products with parameters in $\mathbb F_2$, we abbreviate ``$A\otimes_{\mathbb F_2}B$'' simply by
``$A\otimes B$''. In this way, $T^n(k_1G)$ we will be denoted simply by
$$T^n(k_1G)=k_1G\otimes k_1G\otimes...\otimes k_1G.$$

Next, we have a result that approximates Dickmann-Miraglia's K-theory with the Milnor's reduced K-theory:

\begin{prop}[2.1 \cite{dickmann2006algebraic}]\label{2.1kt}
 Let $G$ be a special group, $x,y,a_1,...,a_n\in G$ and $\sigma$ be a permutation on $n$ elements.
 \begin{enumerate}[a -]
  \item In $k_2G$, $\lambda(a)^2=\lambda(a)\lambda(-1)$. Hence in $k_mG$,
$\lambda(a)^m=\lambda(a)\lambda(-1)^{m-1}$, $m\ge2$;
\item In $k_2G$, $\lambda(a)\lambda(-a)=\lambda(a)^2=0$;
\item In $k_nG$,
$\lambda(a_1)\lambda(a_2)...\lambda(a_n)=\lambda(a_{\sigma 1})\lambda(a_{\sigma
2})...\lambda(a_{\sigma n})$;
\item For $n\ge1$ and $\xi\in k_nG$, $\xi^2=\lambda(-1)^n\xi$;
\item If $G$ is a reduced special group, then $x\in D_G(1,y)$ and $\lambda(y)\lambda(a_1)...\lambda(a_n)=0$ implies
$$\lambda(x)\lambda(a_1)\lambda(a_2)...\lambda(a_n)=0.$$
 \end{enumerate}
\end{prop}

\begin{defn}[2.4 \cite{dickmann2006algebraic}]\label{2.4kt}
$ $
 \begin{enumerate}[a -]
  \item A reduced special group is [MC] if for all $n\le1$ and all forms $\varphi$ over $G$,
  $$\mbox{For all }\sigma\in X_G,\mbox{ if }\sigma(\varphi)\equiv0\,\mbox{mod }2^n\mbox{ then } \varphi\in I^nG.$$
  \item A reduced special group is [SMC] if for all $n\ge1$,  the multiplication by $\lambda(-1)$ is an injection of $k_nG$ in $k_{n+1}G$.
 \end{enumerate}
\end{defn}

\begin{fat}\label{fat1}
We summarize some properties of the Dickmann-Miraglia K-theory below:
\begin{enumerate}[i -]
    \item An inductive system of special groups
$$\mathcal G=(G_i;\{f_{ij}:i\le j\in I\}),$$
provides an inductive system of graded ring, which nodes are $k_*G_i$ and morphisms are
$$(f_{ij})_*:k_*G_i\rightarrow k_*G_j,\,\mbox{for }i\le j\mbox{ in }I.$$

    \item (4.5 of \cite{dickmann2006algebraic})
 Let $\mathcal{G}=(G_i;\{f_{ij}:i,j\in I,\,i\le j\})$ an inductive system of special groups over a directed poset $I$ and $(G;\{f_i:i\in I\})=\varinjlim\mathcal{G}$. Then
 $$k_*G\cong\varinjlim\limits_{i\in I}k_*G_i.$$

    \item (4.6, 5.1, 5,7 and 6.8 of \cite{dickmann2006algebraic}) The inductive limit, finite products, SG-sum and extension of SMC groups is a SMC group.

    \item (5.1 of \cite{dickmann2006algebraic})
 Let $G_1,...,G_m$ be special groups and $\prod^m_{i=1}G_i$. Then there exists a graded isomorphism
$$\gamma:k_*P\rightarrow\bigoplus^m_{i=1}k_*G_i.$$
\end{enumerate}
\end{fat}

\subsection{Multifields/Hyperfields}

Roughly speaking, a multiring is a ``ring'' with a multivalued addition, a notion introduced in the 1950s by Krasner's works. The notion of
multiring was joined to the quadratic forms tools by the hands of M. Marshall in the last decade (\cite{marshall2006real}). We gather the basic information about multirings/hyperfields and expand some details that we use in the context of K-theories. For more detailed calculations involving multirings/hyperfields and quadratic forms we indicate to the reader the reference \cite{ribeiro2016functorial} (or even \cite{worytkiewiczwitt2020witt} and \cite{roberto2021quadratic}). Of course, multi-structures is an entire subject of research (which escapes from the "quadratic context"), and in this sense, we indicate the references \cite{pelea2006multialgebras}, \cite{viro2010hyperfields}, \cite{ameri2019superring}.

\begin{defn}\label{defn:multigroupI}
 A multigroup is a quadruple $(G,\ast,r,1)$, where $G$ is a non-empty set, $\ast:G\times G\rightarrow\mathcal
P(G)\setminus\{\emptyset\}$ and $r:G\rightarrow G$
 are functions, and $1$ is an element of $G$ satisfying:
 \begin{enumerate}[i -]
  \item If $ z\in x\ast y$ then $x\in z\ast r(y)$ and $y\in r(x)\ast z$.
  \item $y\in 1\ast x$ if and only if $x=y$.
  \item With the convention $x\ast(y\ast z)=\bigcup\limits_{w\in y\ast z}x\ast w$ and
  $(x\ast y)\ast z=\bigcup\limits_{t\in x\ast y}t\ast z$,
  $$x\ast(y\ast z)=(x\ast y)\ast z\mbox{ for all }x,y,z\in G.$$

	A multigroup is said to be \textbf{commutative} if
  \item $x\ast y=y\ast x$ for all $x,y\in G$.
 \end{enumerate}
 Observe that by (i) and (ii), $1\ast x=x\ast 1=\{x\}$ for all $x\in G$. When $a\ast b=\{x\}$ be a unitary set, we just write
$a\ast b=x$.
\end{defn}

   \begin{defn}[Adapted from Definition 2.1 in \cite{marshall2006real}]\label{defn:multiring}
 A multiring is a sextuple $(R,+,\cdot,-,0,1)$ where $R$ is a non-empty set, $+:R\times
R\rightarrow\mathcal P(R)\setminus\{\emptyset\}$,
 $\cdot:R\times R\rightarrow R$
 and $-:R\rightarrow R$ are functions, $0$ and $1$ are elements of $R$ satisfying:
 \begin{enumerate}[i -]
  \item $(R,+,-,0)$ is a commutative multigroup;
  \item $(R,\cdot,1)$ is a commutative monoid;
  \item $a.0=0$ for all $a\in R$;
  \item If $c\in a+b$, then $c.d\in a.d+b.d$. Or equivalently, $(a+b).d\subseteq a.d+b.d$.
 \end{enumerate}

Note that if $a \in R$, then $0 = 0.a \in (1+ (-1)).a \subseteq 1.a + (-1).a$, thus $(-1). a = -a$.

 $R$ is said to be an hyperring if for $a,b,c \in R$, $a(b+c) = ab + ac$.

 A multring (respectively, a hyperring) $R$ is said to be a multidomain (hyperdomain) if it has no zero divisors. A multring $R$ will be a
multifield if every non-zero element of $R$ has multiplicative inverse; note that hyperfields and multifields coincide. We will use "hyperfield" since this is the prevailing terminology.
\end{defn}

\begin{defn}\label{defn:morphism}
 Let $A$ and $B$ multirings. A map $f:A\rightarrow B$ is a morphism if for all $a,b,c\in A$:
  \begin{enumerate}[i -]
  \item $c\in a+b\Rightarrow f(c)\in f(a)+f(b)$;
  \item $f(-a)=-f(a)$;
  \item $f(0)=0$;
  \item $f(ab)=f(a)f(b)$;
  \item $f(1)=1$.
 \end{enumerate}
\end{defn}

If $A$ and $B$ are multirings, a morphism $f \colon A \to B$ is a \textit{strong morphism} if for all $a,b, c\in A$, if $f(c) \in f(a) + f(b)$, then there are $a',b',c' \in A $ with $f(a') = f(a),f(b') = f(b), f(c') = f(c)$ such that $c' \in a' + b'$.
 In the quadratic context, there is a more detailed analysis in Example 2.10 of \cite{ribeiro2016functorial}.

 \begin{ex}\label{ex:1.3}
$ $
 \begin{enumerate}[a -]
  \item Suppose that $(G,+,0)$ is an abelian group. Defining $a + b = \{a + b\}$ and $r(g)=-g$,
we have that $(G,+,r,0)$ is an abelian multigroup. In this way, every ring, domain and field is a multiring,
multidomain and hyperfield, respectively.

  \item $Q_2=\{-1,0,1\}$ is hyperfield with the usual product (in $\mathbb Z$) and the multivalued sum defined by
relations
  $$\begin{cases}
     0+x=x+0=x,\,\mbox{for every }x\in Q_2 \\
     1+1=1,\,(-1)+(-1)=-1 \\
     1+(-1)=(-1)+1=\{-1,0,1\}
    \end{cases}
  $$

  \item Let $K=\{0,1\}$ with the usual product and the sum defined by relations $x+0=0+x=x$, $x\in K$ and
$1+1=\{0,1\}$. This is a hyperfield  called Krasner's hyperfield \cite{jun2015algebraic}.
  \end{enumerate}
\end{ex}

 Now, another example that generalizes $Q_2=\{-1,0,1\}$. Since this is a new one, we will provide the entire verification that it is a
multiring:

\begin{ex}[Kaleidoscope, Example 2.7 in \cite{ribeiro2016functorial}]\label{kaleid}
 Let $n\in\mathbb{N}$ and define
 $$X_n=\{-n,...,0,...,n\} \subseteq \mathbb{Z}.$$
 We define the \textbf{$n$-kaleidoscope multiring} by
$(X_n,+,\cdot,-, 0,1)$, where $- : X_n \to X_n$ is restriction of the  opposite map in $\mathbb{Z}$,  $+:X_n\times
X_n\rightarrow\mathcal{P}(X_n)\setminus\{\emptyset\}$ is given by the rules:
 $$a+b=\begin{cases}
    \{a\},\,\mbox{ if }\,b\ne-a\mbox{ and }|b|\le|a| \\
    \{b\},\,\mbox{ if }\,b\ne-a\mbox{ and }|a|\le|b| \\
    \{-a,...,0,...,a\}\mbox{ if }b=-a
   \end{cases},$$
and $\cdot:X_n\times X_n\rightarrow X_n$ is given by the rules:
 $$a\cdot b=\begin{cases}
    \mbox{sgn}(ab)\max\{|a|,|b|\}\mbox{ if }a,b\ne0 \\
    0\mbox{ if }a=0\mbox{ or }b=0
   \end{cases}.$$
  With the above rules we have that $(X_n,+,\cdot, -, 0,1)$ is a multiring.
\end{ex}

 Now, another example that generalizes $K=\{0,1\}$.

\begin{ex}[H-hyperfield, Example 2.8 in \cite{ribeiro2016functorial}]\label{H-multi}
Let $p\ge1$ be a prime integer and $H_p:=\{0,1,...,p-1\} \subseteq \mathbb{N}$. Now, define the binary multioperation and operation in $H_p$ as
follows:
\begin{align*}
 a+b&=
 \begin{cases}H_p\mbox{ if }a=b,\,a,b\ne0 \\ \{a,b\} \mbox{ if }a\ne b,\,a,b\ne0 \\ \{a\} \mbox{ if }b=0 \\ \{b\}\mbox{ if }a=0 \end{cases} \\
 a\cdot b&=k\mbox{ where }0\le k<p\mbox{ and }k\equiv ab\mbox{ mod p}.
\end{align*}
$(H_p,+,\cdot,-, 0,1)$ is a hyperfield such that for all $a\in H_p$, $-a=a$. In fact, these $H_p$ are a kind of generalization of $K$, in the sense that $H_2=K$.
\end{ex}

 There are many natural constructions on the category of multrings as: products, directed inductive limits, quotients by an ideal,
localizations by multiplicative subsets and quotients by ideals.  Now, we present some constructions that will be used further. For the first one, we need to restrict our category:

\begin{defn}[Definition 3.1 of \cite{roberto2021quadratic}]\label{hiperb}
An \textbf{hyperbolic multiring} is a multiring $R$ such that $1-1=R$. The category of hyperbolic multirings and hyperbolic hyperfields will be denoted by $\mathcal{HMR}$ and $\mathcal{HMF}$ respectively.
\end{defn}

Let $F_1$ and $F_2$ be two hyperbolic hyperfields. We define a new hyperbolic hyperfield $(F_1\times_h F_2,+,-,\cdot,(0,0),(1,1))$ by the following: the adjacent set of this structure is
$$F_1\times_h F_2:=(\dot F_1\times \dot F_2)\cup\{(0,0)\}.$$
For $(a,b),(c,d)\in F_1\times_h F_2$ we define
\begin{align}\label{prodmultiop}
    -(a,b)&=(-a,-b),\nonumber\\
    (a,b)\cdot(c,d)&=(a\cdot c,b\cdot d),\nonumber \\
    (a,b)+(c,d)&=\{(e,f)\in F_1\times F_2:e\in a+c\mbox{ and }f\in b+d\}\cap(F_1\times_h F_2).
\end{align}
In other words, $(a,b)+(c,d)$ is defined in order to avoid elements of $F_1\times F_2$ of type $(x,0),(0,y)$, $x,y\ne0$.

\begin{teo}[Product of Hyperbolic Hyperfields]\label{hfproduct}
Let $F_1,F_2$ be hyperbolic hyperfields and $F_1\times_h F_2$ as above. Then $F_1\times_h F_2$ is a hyperbolic hyperfield and satisfy the Universal Property of product for $F_1$ and $F_2$.
\end{teo}
\begin{proof}
 We will verify the conditions of definition \ref{defn:multiring} (in a very similar manner as in Theorem 3.3 of \cite{ribeiro2016functorial}). Note that by the definition of multivalued sum in $F_1\times_h F_2$ we have for all $a,c\in F_1$ and all $b,d\in F_2$ that $(a,b)+(c,d)\ne\emptyset$ and $(a,c)-(a,c)=F_1\times_h F_2$ if $a,c\ne0$.
\begin{enumerate}[i -]
 \item In order to prove that $(F_1\times_h F_2,+,-,(0,0))$ is a multigroup we follow the steps below.
 \begin{enumerate}[a -]
  \item We will prove that if $(e,f)\in(a,b)+(c,d)$, then $(a,b)\in(c,d)+ (-e,-f)$ and $(c,d)\in (-a,-b)+(e,f)$.

  If $(a,b)=(0,0)$ (or $(c,d)=(0,0)$) or $(a,b)=(-c,-d)$, then $(e,f)\in(a,b)+(c,d)$ means $(e,f)=(c,d)$ or $e\in a-a$, $f\in b-b$. In both cases we get $(a,b)\in(c,d)+(-e,-f)$ and $(c,d)\in (-a,-b)+(e,f)$.

  Now suppose $a,b,c,d\ne0$ with $a\ne-c$, $b\ne-d$. Let $(e,f)\in(a,b)+(c,d)$. Then $(e,f)\in F_1\times_hF_2$ with $e\in a+c$ and $f\in c+d$. Moreover $a\in c-e$ and $b\in d-f$ with $(a,b)\in F_1\times_hF_2$, implying $(a,b)\in(c,d)+ (-e,-f)$. We prove $(c,d)\in (-a,-b)+(e,f)$ by the very same argument.

  \item Commutativity and $((a,b)\in(c,d)+(0,0))\Leftrightarrow((a=b)\wedge (c=d))$ are direct consequence of the definition of multivaluated sum.

  \item Now we prove the associativity, that is,
  $$[(a,b)+(c,d)]+(e,f)=(a,b)+[(c,d)+(e,f)].$$
  In fact (see the remarks after Lemma 2.4 of \cite{ribeiro2016functorial}), it is enough to show
  $$[(a,b)+(c,d)]+(e,f)\subseteq(a,b)+[(c,d)+(e,f)].$$

  If $(a,b)=0$ or $(c,d)=0$ or $(e,f)=0$ we are done. Now let $a,b,c,d,e,f\ne0$ and $(v,w)\in[(a,b)+(c,d)]+(e,f)$. If $(c,d)=-(e,f)$, we have
  $$(a,b)+[(c,d)+(e,f)]=(a,b)+F_1\times_hF_2=F_1\times_hF_2\supseteq [(a,b)+(c,d)]+(e,f).$$
  If $-(e,f)\in(a,b)+(c,d)$ then $-(a,b)\in(c,d)+(e,f)$ and we have
      $$[(a,b)+(c,d)]+(e,f)=F_1\times_hF_2=(a,b)+[(c,d)+(e,f)]$$

  Now suppose $a,b,c,d,e,f\ne0$, $(c,d)\ne-(e,f)$, $-(e,f)\notin(a,b)+(c,d)$. Let $(x,y)\in[(a,b)+(c,d)]+(e,f)$. Then there exists $(v,w)\in F_1\times_hF_2$ such that $(v,w)\in(a,b)+(c,d)$ and $(x,y)\in(v,w)+(e,f)$. This imply $(v\in a+c)\wedge(x\in v+e)$ and $(w\in b+y)\wedge(y\in d+f)$, so there exists $p\in F_1$, $q\in F_2$ such that
  $(v\in a+p)\wedge(p\in c+e)$ and $(w\in b+q)\wedge(q\in d+f)$. If $p,q=0$ or $p,q\ne0$ we have $(p,q)\in F_1\times_hF_2$, which imply $(v,w)\in [(a,b)+(c,d)]+(e,f)$. If $p=0$ and $q\ne0$ (the case $q=0$ and $p\ne0$ is analogous), then $v=a$ and $c=-e$. Since $a,c\ne0$ and $F_1$ is hyperbolic, we have $a-a=c-c=F_1$. Then $(v\in a-a)\wedge(-a\in c-c)$ and $(w\in b+q)\wedge(q\in d+f)$, with $(-a,q)\in F_1\times_hF_2$ and again, we get $(v,w)\in [(a,b)+(c,d)]+(e,f)$.
 \end{enumerate}

   \item Since $(F_1\times_h F_2,\cdot,(1,1))$ is an abelian group, we conclude that $(F_1\times_h F_2,\cdot,(1,1))$ is a commutative monoid. Beyond this, every nonzero element of $F_1\times_h F_2$ has an inverse.

  \item $(a,b)\cdot(0,0)=(0,0)$ for all $(a,b)\in F_1\times_h F_2$ is direct from definition.

  \item For the distributive property, let $(a,b),(c,d),(e,f)\in F_1\times_h F_2$ and consider $(x,y)\in(e,f)[(a,b)+(c,d)]$. We need to prove that
  \begin{align*}
      \tag{*}(x,y)\in(e,f)\cdot(a,b)+(e,f)\cdot(c,d).
  \end{align*}
  It is the case if $(a,b)=(0,0)$, or $(c,d)=(0,0)$ or $(e,f)=(0,0)$. Moreover (*) also holds if $a,b,c,d,e,f\ne0$ and $(c,d)=-(a,b)$.

  Now suppose $a,b,c,d,e,f\ne0$ and $(c,d)\ne-(a,b)$. Then $(x,y)=(ev,fw)$ for some $(v,w)\in(a,b)+(c,d)$. Since $(e,f)\in F_1\times_h F_2$, $e=0$ iff $f=0$, which imply $(ev,fw)\in F_1\times_h F_2$, with $ev\in ea+ec$ and $fw\in fb+fd$. Therefore $(x,y)=(ev,fw)\in(e,f)\cdot(a,b)+(e,f)\cdot(c,d)$, as desired.
\end{enumerate}
Then $(F_1\times_h F_2,+,-,\cdot,(0,0),(1,1))$ is a hyperbolic hyperfield. Moreover we have projections $\pi_1:F_1\times_h F_2\rightarrow F_1$, $\pi_2:F_1\times_h F_2\rightarrow F_2$ given respectively by the rules $\pi_1(x,y)=x$, $\pi_2(x,y)=y$.

Finally, suppose that $F$ is another hyperfield with morphisms $p_1:F\rightarrow F_1$, $p_2:F\rightarrow F_2$. Consider $(p_1,p_2):F\rightarrow F_1\times_h F_2$ given by the rule
$(p_1,p_2)(x)=(p_1(x),p_2(x))$. It is immediate that $(p_1,p_2)$ is the unique morphism making the following diagram commute
$$\xymatrix@!=5.pc{ & F\ar[dr]^{p_2}\ar[dl]_{p_1}\ar@{.>}[d]^{(p_1,p_2)} & \\
F_1 & F_1\times_hF_2\ar[r]_{p_2}\ar[l]^{p_1} & F_2}$$
so $F_1\times_h F_2$ is the product in the category of hyperbolic hyperfields, completing the proof.
\end{proof}

In order to avoid confusion and mistakes, we denote the binary product in $\mathcal{HMF}$ by $F_1\times_hF_2$. For hyperfields $\{F_i\}_{i\in I}$, we denote the product of this family by
$$\prod^h_{i\in I}F_i,$$
with underlying set defined by
$$\prod^h_{i\in I}F_i:=\left(\prod_{i\in I}\dot F_i\right)\cup\{(0_i)_{i\in I}\}$$
and operations defined by rules similar to the ones defined in \ref{prodmultiop}. If $I=\{1,...n\}$, we denote
$$\prod^h_{i\in I}F_i=\prod^n_{\substack{i=1 \\ [h]}}F_i.$$

\begin{ex}
Note that if $F_1$ (or $F_2$) is not hyperbolic, then $F_1\times_h F_2$ is not a hyperfield. Let $F_1$ be a field (considered as a hyperfield), for example $F_1=\mathbb R$ and $F_2$ be another hyperfield. Then if $a,b\in F_2$, we have
$1-1=\{0\}$, so $(1,a)+(-1,b)=\{0\}\times(a-b)$, and $$[\{0\}\times(a-b)]\cap(F_1\times_h F_2)=\emptyset.$$
\end{ex}

 \begin{prop}[Example 2.6 in \cite{marshall2006real}]\label{defn:strangeloc}
 Fix a multiring $A$ and a multiplicative subset $S$ of $A$ such that $1\in S$. Define an equivalence relation $\sim$
on $A$ by $a\sim b$ if and only if $as=bt$ for some $s,t\in S$. Denote by $\overline a$ the equivalence class of
$a$ and set $A/_mS=\{\overline a:a\in A\}$. Then, we define in agreement with Marshall's notation, $\overline a+\overline
b=\{\overline c:cv\in as+bt,\,\mbox{for some }s,t,v\in S\}$, $-\overline a=\overline{-a}$, and
$\overline{a}\overline{b}=\overline{ab}$.

Then  $A/_mS$ are multirings. Moreover, if $A$ is a hyperring, the same holds for $A/_mS$. The canonical projection $\pi:A\rightarrow A/_mS$ is
a morphism.
 \end{prop}

 \begin{prop}[2.19 in \cite{ribeiro2016functorial}]
  Let $A,B$ be a multiring and $S\subseteq A$ a multiplicative subset of $A$. Then for every morphism
$f:A\rightarrow B$ such that $f[S]=\{1\}$, there exist a unique morphism $\tilde f:A/_mS\rightarrow B$ such
that the following diagram commute:
$$\xymatrix{A\ar[r]^{\pi}\ar[dr]_{f} & A/_mS\ar[d]^{!\tilde f} \\ & B}$$
where $\pi:A\rightarrow A/_mS$ is the canonical projection $\pi(a)=\overline a$.
 \end{prop}

\begin{prop}[3.13 of \cite{ribeiro2016functorial}]\label{sg.to.mf}
 Let $(G,\equiv,-1)$ be a special group and define $M(G)=G\cup\{0\}$ where $0:=\{G\}$\footnote{Here,
the choice of the zero element was ad hoc. Indeed, we can define $0:=\{x\}$ for any $x\notin G$.}. Then
$(M(G),+,-,\cdot,0,1)$ is a hyperfield, where
\begin{itemize}
   \item $a\cdot b=\begin{cases}0\,\mbox{if }a=0\mbox{ or }b=0 \\ a\cdot
b\,\mbox{otherwise}\end{cases}$
   \item $-(a)=(-1)\cdot a$
   \item $a+b=\begin{cases}\{b\}\,\mbox{if }a=0 \\ \{a\}\,\mbox{if }b=0\\ M(G)\,\mbox{if
}a=-b,\,\mbox{and }a\ne0
\\
D_G(a,b)\,\mbox{otherwise}\end{cases}$
  \end{itemize}
\end{prop}

\begin{cor}[3.14 of \cite{ribeiro2016functorial}]\label{cor:equiv1}
 The correspondence $G\mapsto M(G)$ extends to a faithful functor $M:\mathcal{SG}\rightarrow MField$.
\end{cor}

\begin{defn}[3.15-3.19 of \cite{ribeiro2016functorial}]\label{smf}
A hyperfield $F$ is a \textbf{special multifield} if there exist a special group $G$ such that $F=M(G)$. The category of special multifields will be denoted by $\mathcal{SMF}$.
\end{defn}

\begin{defn}[Definition 3.2 of \cite{roberto2021quadratic}]
A \textbf{Dickmann-Miraglia multiring (or DM-multiring for short)} \footnote{The name ``Dickmann-Miraglia'' is given in honor to professors Maximo Dickmann and Francisco Miraglia, the creators of the special group theory.} is a pair $(R,T)$ such that $R$ is a multiring, $T\subseteq R$ is a multiplicative subset of $R\setminus\{0\}$, and $(R,T)$ satisfies the following properties:
\begin{description}
\item [DM0] $R/_mT$ is hyperbolic.
 \item [DM1] If $\overline{a}\ne0$ in $R/_mT$, then $\overline a^2=\overline 1$ in $R/_mT$. In other words, for all $a\in R\setminus\{0\}$,
there are $r,s\in T$ such that $ar=s$.
 \item [DM2] For all $a\in R$, $(\overline 1-\overline a)(\overline 1-\overline a)\subseteq(\overline 1-\overline a)$ in
$R/_mT$.
 \item [DM3] For all $a,b,x,y,z\in R\setminus\{0\}$, if
 $$\begin{cases}\overline a\in \overline x+\overline b \\ \overline b\in \overline y+\overline z\end{cases}\mbox{ in }R/_mT,$$
 then exist $\overline v\in\overline x+\overline z$ such that $\overline a\in\overline y+\overline
v$ and $\overline{vb}\in\overline{xy}+\overline{az}$ in $R/_mT$.
\end{description}

If $R$ is a ring, we just say that $(R,T)$ is a DM-ring, or $R$ is a DM-ring. A Dickmann-Miraglia hyperfield (or DM-hyperfield) $F$ is a
hyperfield such that $(F,\{1\})$ is a DM-multiring (satisfies DM0-DM3). In other words, $F$ is a DM-hyperfield if $F$ is hyperbolic and for all
$a,b,v,x,y,z\in F^*$,
\begin{enumerate}[i -]
 \item $a^2=1$.
 \item $(1-a)(1-a)\subseteq(1-a)$.
 \item $\mbox{If }\begin{cases}a\in x+b \\ b\in y+z\end{cases}\mbox{ then there exists }v\in x+z\mbox{ such that }a\in y+v\mbox{ and }vb\in xy+az$.
\end{enumerate}
\end{defn}

\begin{teo}[Theorem 3.4 of \cite{roberto2021quadratic}]\label{teopmf}
 Let $(R,T)$ be a DM-multiring and  denote
 $$Sm(R,T)=(R/_mT).$$
 Then $Sm(R)$ is a special hyperfield (thus $Sm(R,T)^\times$ is a special group).
\end{teo}

\begin{teo}[Theorem 3.9 of \cite{roberto2021quadratic}]
 Let $F$ be a hyperfield satisfying DM0-DM2. Then $F$ satisfies DM3 if and only if satisfies SMF4. In other words, $F$ is a DM-hyperfield if and only if it is a special hyperfield.
\end{teo}

In this sense, we define the following category:

\begin{defn}
A \textbf{pre-special hyperfield} is a hyperfield satisfying DM0, DM1 and DM2. In other words, a pre-special hyperfield is a hyperbolic hyperfield $F$ such that for all $a\in\dot F$, $a^2=1$ and $(1-a)(1-a)\subseteq1-a$.

The category of pre-special hyperfields will be denoted by $PSMF$.
\end{defn}

\begin{ex}[An hyperfield satisfying DM3 but not DM2]\label{expresgn}
Let $H_3$ as in Example \ref{H-multi}. We have that $H_3$ is hyperbolic and has exponent 2 but does not satisfies $DM2$:
$$1+2=\{1,2\}\mbox{ and }(1+2)(1+2)=1+2+2+1=(1+1)+(2+2)=H_3.$$
Moreover, by a case analysis we have that $H_3$ satisfies DM3, and we prove it by a case analysis: we have 27 possible choices for the triple $(x,y,z)$ with $x,y,z\in H_3$.

\begin{enumerate}[i -]
    \item Let $(x,y,z)$ such that $z=0$. Then $b\in y+z$ means $b=y$ and $v\in x+z$ means $v=x$. Let $a\in x+b$ with $b\in y+z$. Then $a\in b+x=y+v$ and
    $$vb=xy\in xy+az.$$
    So we cover the cases
    $$(x,y,z)\in\{(0,0,0),(0,1,0),(0,2,0),(1,0,0),(1,1,0),(1,2,0),(2,0,0),(2,1,0),(2,2,0)\}.$$

    \item Let $(x,y,z)$ such that $x=0$. Then $v\in x+z$ means $v=z$. Let $a\in x+b=0+b=\{b\}$ with $b\in y+z$. Then $a\in b+z=y+z$, $az=bz$ and
    $$vb=bz\in xy+bz=xy+az.$$
    So we cover the cases
    $$(x,y,z)\in\{(0,0,1),(0,0,2),(0,1,1),(0,1,2),(0,2,1),(0,2,2)\}.$$

    \item Let $(x,y,z)$ such that $y=0$ and $x,z\ne0$. Then $\{1,2\}\subseteq x+z$, $b\in y+z$ means $b=z$ and $a\in y+v$ means $a=v$. Let $a\in x+b=x+z$. Then $vb=az$ and
    $$vb=az\in0+az=xy+az.$$
    So we cover the cases
    $$(x,y,z)\in\{(1,0,1),(1,0,2),(2,0,1),(2,0,2)\}.$$

    \item Let $(x,y,z)$ such that $x,y,z\ne0$ and $x+z=H_3$. Let $a\in x+b$ with $b\in y+z$. If $a=0$ then $b=x$. Taking $v=y\in H_3=x+z$ we get $a\in y+y=y+v$ and
    $$vb=xy\in xy+0=xy+az.$$
    If $b=0$ then $z=y$ $a=x$ (so $az=xy$). Taking $v=y\in H_3=x+z$ we get $a\in y+y=y+v$ and
    $$vb=0\in xy+xy=xy+az.$$
    If $a,b\ne0$, then there exist $v\in H_3=x+z$ such that $vb=xy$. Then $a\in\{1,2\}\subseteq y+v$ (because $x,y,z,v\ne0$) and
    $$vb=xy\in xy+az.$$
    So we cover the cases
    $$(x,y,z)\in\{(1,1,1),(1,2,1),(2,1,2),(2,2,2)\}.$$

    \item Let $(x,y,z)$ such that $x,y,z\ne0$ and $y+z=H_3$. Let $a\in x+b$ with $b\in y+z$. If $a=0$ then $b=x$. Taking $v=y\in\{1,2\}\subseteq x+z$ we get $a\in y+y=y+v$ and
    $$vb=xy\in xy+0=xy+az.$$
    If $b=0$ then $z=y$ $a=x$ (so $az=xy$). Taking $v=y\in \{1,2\}\subseteq x+z$ we get $a\in y+y=y+v$ and
    $$vb=0\in xy+xy=xy+az.$$
    If $a,b\ne0$, then there exist $v\in\{1,2\}\subseteq x+z$ such that $vb=xy$. Then $a\in\{1,2\}\subseteq y+v$ (because $x,y,z,v\ne0$) and
    $$vb=xy\in xy+az.$$
    So we cover the cases
    $$(x,y,z)\in\{(1,1,2),(1,2,2),(2,1,1),(2,2,1)\}.$$
\end{enumerate}
\end{ex}

Example \ref{expresgn} is a big surprise in the sense that, in the context of special groups, this example exhibits a group that it is not pre-special but satisfies SG6.

  \section{The K-theory for Multifields/Hyperfields}

In this section we introduce the notion of K-theory of a hyperfield essentially repeating the construction in \ref{milkt} replacing the word ``field'' by ``hyperfield'' and explore some of this basic properties. In particular, Theorem \ref{fixsg3}  is an extension of a result \cite{wadsworth55merkurjev}, that gives us some evidence, that apart from the obvious resemblance, more technical aspects of this new theory can be developed (but with other proofs) in multistructure setting in parallel with classical K-theory.

    \begin{defn}[The K-theory of a Hyperfield]
  For a hyperfield $F$, $K_*F$ is the graded ring
$$K_*F=(K_0F,K_1F,K_2F,...)$$
  defined by the following rules: $K_0F:=\mathbb Z$. $K_1F$ is the multiplicative group $\dot F$ written additively.
  With this purpose, we fix the canonical ``logarithm'' isomorphism
$$\rho:\dot F\rightarrow K_1F,$$
  where $\rho(ab)=\rho(a)+\rho(b)$. Then $K_nF$ is defined to be the quotient of the tensor algebra
  $$K_1F\otimes K_1F\otimes...\otimes K_1F\,(n \mbox{ times})$$
  by the (homogeneous) ideal generated by all $\rho(a)\otimes \rho(b)$, with $a\ne0,1$ and $b\in1-a$.
  \end{defn}

  In other words, for each $n\ge2$,
$$K_nF:=T^n(K_1F)/Q^n(K_1(F)),$$
where
$$T^n(K_1F):=K_1F\otimes_{\mathbb Z} K_1F\otimes_{\mathbb Z}...\otimes_{\mathbb Z} K_1F$$
and $Q^n(K_1(F))$ is the subgroup generated by all expressions of type $\rho(a_1)\otimes\rho(a_2)\otimes...\otimes\rho(a_n)$ such that $a_i\in1-a_j$ for some $i,j$ with $1\le i,j\le n$.

  To avoid carrying the overline symbol, we will adopt all the conventions used in Dickmann-Miraglia's K-theory (as explained in above definition
\ref{defn:ksg}). Just as it happens with the previous K-theories, a generic element $\eta\in K_nF$ has the pattern
  $$\eta=\rho(a_1)\otimes\rho(a_2)\otimes...\otimes\rho(a_n)$$
  for some $a_1,...,a_n\in\dot F$, with $a_i\in1-a_j$ for some $1\le i<j\le n$. Note that if $F$ is a field, then ``$b\in1-a$'' just means
$b=1-a$, and the hyperfield and Milnor's K-theory for $F$ coincide.

  The very first task, is to extend the basic properties valid in Milnor's and Dickmann-Miraglia's K-theory to ours. Here we already need to
restrict our attention to hyperbolic hyperfields:

  \begin{lem}[Basic Properties I]\label{bp1}
 Let $F$ be an hyperbolic hyperfield. Then
 \begin{enumerate}[a -]
 \item $\rho(1)=0$.
 \item For all $a\in\dot F$, $\rho(a)\rho(-a)=0$ in $K_2F$.
 \item For all $a,b\in\dot F$, $\rho(a)\rho(b)=-\rho(a)\rho(b)$ in $K_2F$.
 \item  For every $a_1,...,a_n\in\dot F$ and every permutation $\sigma\in S_n$,
 $$\rho(a_1)...\rho(a_i)...\rho(a_n)=\mbox{sgn}(\sigma)\rho(a_1)...\rho(a_n)\mbox{ in }K_nF.$$
  \item For every $\xi\in K_mF$ and $\eta\in K_nF$, $\eta\xi=(-1)^{mn}\xi\eta$ in $K_{m+n}F$.
 \item For all $a\in\dot F$, $\rho(a)^2=\rho(a)\rho(-1)$.
 \end{enumerate}
\end{lem}
\begin{proof}
 $ $
 \begin{enumerate}[a -]
  \item Is an immediate consequence of the fact that $\rho$ is an isomorphism.
  \item Since $F$ hiperbolic, $1-1=F$. Then $-a^{-1}\in1-1$ for all $a\in\dot F$, and hence, $-1\in-1+a^{-1}$. Multiplying this by
$-a$, we get $a\in1-a$. By definition, this imply $\rho(a)\rho(-a)=0$.

  \item By item (b), $\rho(ab)\rho(-ab)=0$ in $K_2F$. But
  \begin{align*}
   \rho(ab)\rho(-ab)&=\rho(a)\rho((-a)b)+\rho(b)\rho((-b)a) \\
   &=\rho(a)\rho(-a)+\rho(a)\rho(b)+\rho(b)\rho(-b)+\rho(b)\rho(a) \\
   &=\rho(a)\rho(b)+\rho(b)\rho(a).
  \end{align*}
  From $\rho(a)\rho(b)+\rho(b)\rho(a)=\rho(ab)\rho(-ab)=0$, we get the desired result $\rho(a)\rho(b)=-\rho(a)\rho(b)$ in $K_2F$.

  \item This is a consequence of item (c) and an inductive argument.

  \item This is a consequence of item (d) and an inductive argument, using the fact that an element in $K_nF$ has pattern
  $$\eta=\rho(a_1)\otimes \rho(a_2)\otimes...\otimes \rho(a_n)$$
  for some $a_1,...,a_n\in\dot F$, with $a_i\in1-a_j$ for some $1\le i<j\le n$.

 \item Follow from the fact that $F$ is hyperbolic i.e, for all $a\in\dot F$, $a\in1-1$.
 \end{enumerate}
\end{proof}

An element $a\in\dot F$ induces a morphism of graded rings $\omega^a=\{\omega^a_n\}_{n\ge1}:K_*F\rightarrow K_*F$ of degree 1, where $\omega^a_n:K_nF\rightarrow K_{n+1}F$ is the multiplication by $\lambda(-a)$. When $a=-1$, we write
$$\omega=\{\omega_n\}_{n\ge1}=\{\omega^{-1}_n\}_{n\ge1}=\omega^{-1}.$$

 \begin{prop}[Adapted from 3.3 of \cite{dickmann2006algebraic}]\label{3.3ktmultiadap}
 Let $F,K$ be hyperbolic hyperfields and $\varphi:F\rightarrow L$ be a morphism. Then $\varphi$ induces a morphism of graded rings
 $$\varphi_*=\{\varphi_n:n\ge0\}:K_*F\rightarrow K_*L,$$
 where $\varphi_0=Id_{\mathbb Z}$ and for all $n\ge1$, $\varphi_n$ is given by the following rule on generators
$$\varphi_n(\rho(a_1)...\rho(a_n))=\rho(\varphi(a_1))...\rho(\varphi(a_n)).$$
Moreover if $\varphi$ is surjective then $\varphi_*$ is also surjective, and if $\psi:L\rightarrow M$ is another morphism then
\begin{enumerate}[a -]
 \item $(\psi\circ\varphi)_*=\psi_*\circ\varphi_*$ and $Id_*=Id$.
 \item For all $a\in \dot F$ the following diagram commute:
 $$\xymatrix@!=5pc{K_nF\ar[d]_{\varphi_n}\ar[r]^{\omega^a_n} & K_{n+1}F\ar[d]^{\varphi_{n+1}} \\ K_nL\ar[r]_{\omega^{\varphi(a)}_n} & K_{n+1}L}$$
 \item If $\varphi(1)=1$ then for all $n\ge1$ the following diagram commute:
 $$\xymatrix@!=5pc{K_nF\ar[d]_{\varphi_n}\ar[r]^{\omega^{-1}_n} & K_{n+1}F\ar[d]^{\varphi_{n+1}} \\ K_nL\ar[r]_{\omega^{-1}_n} & K_{n+1}L}$$
\end{enumerate}
\end{prop}
\begin{proof}
Firstly, note that $\varphi$ extends to a function $\varphi_1:K_1F\rightarrow K_1L$ given by the rule
$$\varphi_1(\rho(a))=\rho(\varphi(a)).$$

Certainly $\varphi_1$ is a morphism because
$$\varphi_1(0)=\varphi_1(\rho(1))=\rho(\varphi(1))=\rho(1)=0,$$
and for all $\rho(a),\rho(b)\in K_1F$,
$$\varphi_1(\rho(a)+\rho(b))=\varphi_1(\rho(ab))=\rho(\varphi(ab))=\rho(\varphi(a)\varphi(b))=\rho(\varphi(a))+\rho(\varphi(b)).$$
Proceeding inductively, for all $n\ge1$ we extend $\varphi$ to a function
$\varphi_n:\prod^n_{i=1}K_1F\rightarrow K_nL$ given by the rule
$$\varphi(\rho(a_1),...,\rho(a_n)):=\varphi_1(\rho(a_1))...\varphi_1(\rho(a_n))=\rho(\varphi(a_1))...\rho(\varphi(a_n)).$$
Then if $i=1,...,n$ and $b_i\in k_1F$ we have
\begin{align*}
 &\varphi_n(\rho(a_1),...,\rho(a_i)+\rho(b_i),...,\rho(a_n))=\varphi_n(\rho(a_1),...,\rho(a_ib_i),...,\rho(a_n))= \\
 &\rho(\varphi(a_1))...\rho(\varphi(a_ib_i)...\rho(\varphi(a_n))=\rho(\varphi(a_1))...\rho(\varphi(a_i)\varphi(b_i))...\rho(\varphi(a_n))= \\
 &\rho(\varphi(a_1)...[\rho(\varphi(a_i)+\varphi(b_i))]...\rho(\varphi(a_n))= \\
 &\rho(\varphi(a_1))...\rho(\varphi(a_i))...\rho(\varphi(a_n))+\rho(\varphi(a_1))...\rho(\varphi(b_i))...\rho(\varphi(a_n))= \\
 &\varphi_n(\rho(a_1),...,\rho(a_i),...,\rho(a_n))
 +\varphi_n(\rho(a_1),...,\rho(b_i),...,\rho(a_n)),
\end{align*}
then for each $n$, $\varphi_n:\prod^n_{i=1}K_1F\rightarrow K_nL$ is multilinear and by the universal property of tensor product there is an unique morphism
$$\tilde\varphi_n:\bigotimes^n_{j=1}K_1F\rightarrow K_nL$$
extending $\varphi_n$. By construction (and using the fact that $\varphi$ is a morphism), $\mbox{Ker}(\tilde \varphi_n)=Q^n(K_1F)$, which provides an unique morphism
$\overline \varphi_n:T^n(K_1F)/Q^n(K_1(F)\rightarrow K_nL$ such that $\tilde \varphi_n=\overline\varphi_n\circ\pi_n$, where $\pi_n$ is the canonical projection $T^n(K_1F)$ in $Q^n(k_1F)$. Then taking $\varphi_0=Id_{\mathbb Z}$, we get a morhism $\varphi_*:K_*F\rightarrow K_*L$, given by $\varphi_*=\{\overline \varphi_n:n\ge0\}$.

For items (a) and (b), it is enough to note that these properties holds for $\tilde \varphi_n$, $n\ge0$, and after the application of projection, we get the validity for $\overline \varphi_n=\pi_n\circ\tilde \varphi_n$.

Item (c) follows by the same argument of itens (a) and (b), noting that $\varphi(1)=1$ imply $\varphi(-1)=-1$. By abuse of notation, we denote
$$\varphi_*=\{\overline \varphi_n:n\ge0\}=\{\varphi_n:n\ge0\}.$$
\end{proof}

  We also have the reduced K-theory graded ring  $k_*F=(k_0F,k_1F,...,k_nF,...)$ in the hyperfield context, which is defined by the rule $k_nF:=K_nF/2K_nF$ for all $n\ge0$. Of course
  for all $n\ge0$ we have an epimorphism $q:K_nF\rightarrow k_nF$ simply denoted by $q(a):=[a]$, $a\in K_nF$. It is immediate that $k_nF$ is additively generated by $\{[\rho(a_1)]..[\rho(a_n)]:a_1,...,a_n\in\dot F\}$. We simply denote such a generator by $\tilde\rho(a_1)...\tilde\rho(a_n)$ or even $\rho(a_1)...\rho(a_n)$ whenever the context allows it.

  We also have some basic properties of the reduced K-theory, which proof is just a translation of 2.1 of \cite{dickmann2006algebraic}:

  \begin{lem}[Adapted from 2.1 \cite{dickmann2006algebraic}]\label{2.1ktmulti}
 Let $F$ be a hyperbolic hyperfield, $x,y,a_1,...,a_n\in\dot F$ and $\sigma$ be a permutation on $n$ elements.
 \begin{enumerate}[a -]
  \item In $k_2F$, $\rho(a)^2=\rho(a)\rho(-1)$. Hence in $k_mF$,
$\rho(a)^m=\rho(a)\rho(-1)^{m-1}$, $m\ge2$;
\item In $k_2F$, $\rho(a)\rho(-a)=\rho(a)^2=0$;
\item In $k_nF$,
$\rho(a_1)\rho(a_2)...\rho(a_n)=\rho(a_{\sigma 1})\rho(a_{\sigma
2})...\rho(a_{\sigma n})$;
\item For $n\ge1$ and $\xi\in k_nF$, $\xi^2=\rho(-1)^n\xi$;
\item If $F$ is a real reduced hyperfield, then $x\in1+y$ and $\rho(y)\rho(a_1)...\rho(a_n)=0$ implies
$$\rho(x)\rho(a_1)\rho(a_2)...\rho(a_n)=0.$$
 \end{enumerate}
\end{lem}

  Moreover the results in Proposition \ref{3.3ktmultiadap} continue to hold if we took $\varphi_*=\{\varphi_n:n\ge0\}:k_*F\rightarrow k_*L$.

  \begin{prop}\label{ktmarshall1}
  Let $F$ be a hyperfield and $T\subseteq F$ be a multiplicative subset such that $F\subseteq T$. Then
  $$K(F/_m T^*)\cong k(F/_mT^*).$$
  \end{prop}
  \begin{proof}
  Since $F^2\subseteq T$, for all $a\in F/_mT^*$ we have
  $$0=\rho(a^2)=\rho(a)+\rho(a).$$
  Then $2K(F/_mT^*)=0$ and we get $K(F/_m T^*)\cong k(F/_mT^*)$.
  \end{proof}

  \begin{teo}\label{ktmarshall2}
  Let $F$ be a hyperbolic hyperfield and $T\subseteq F$ be a multiplicative subset such that $F\subseteq T$. Then there is a surjective morphism
  $$k(F)\rightarrow k(F/_mT^*).$$
  Moreover,
  $$k(F)\cong K(F/_m \dot F^2)\cong k(F/_m\dot F^2).$$
  \end{teo}

  Berfore we prove it, we need a Lemma:
 \begin{lem}\label{lemktmulti1}
 Let $F$ be a hyperfield and $n\ge1$. Then
 $$2K_n(F)=\left\lbrace\sum^p_{j=1}\rho(a_{j1})...\rho(a_{jn}):\mbox{for all }j\mbox{ there is an index }k\mbox{ such that }a_{jk}=b_i^2,\,b_i\in\dot F\right\rbrace.$$
 \end{lem}
 \begin{proof}
 Let $\eta\in2K_nF$. Then
 $$\eta=\left(\sum^p_{j=1}\rho(a_{j1})...\rho(a_{jn})\right)+\left(\sum^p_{j=1}\rho(a_{j1})...\rho(a_{jn})\right),\,d_{ij}\in\dot F.$$
 By induction, we only need to consider the case $p=1$, so
 \begin{align*}
     \rho(a_1)...\rho(a_n)+\rho(a_1)...\rho(a_n)&=
     \rho(a_1^2)\rho(a_2)...\rho(a_n).
 \end{align*}
 and we get $\subseteq$. The reverse inclusion follow by the same calculation.
 \end{proof}

  \begin{proof}[Proof of Theorem \ref{ktmarshall2}]
  Let $\pi:F\rightarrow F/_mT^*$ denote the canonical projection. By Proposition \ref{3.3ktmultiadap} there is a morphism $\pi_*:K(F)\rightarrow K(F/_mT^*)$. Since $\pi$ is surjective, $\pi_*$ is surjective.

  Now, let $\pi:F\rightarrow F/_m\dot F^2$ and $q:K(F)\rightarrow k(F)$ the  canonical projections. Denote elements in $F/_m\dot F^2$ by $[a]\in F/_m\dot F^2$, $a\in F$ and elements in $k_n(F)$ by
  $\tilde\rho(a_1)...\tilde\rho(a_n)$. For all $n\ge1$ we have an induced morphism $\tilde q_n:K_n(F/_m\dot F^2)\rightarrow k_n(F)$ given by the rule
  $$\tilde q_n(\rho([a_1])...\rho([a_n])):=\tilde\rho(a_1)...\tilde\rho(a_n).$$
  This morphism $\tilde\pi_n$ makes the following diagram commute
  $$\xymatrix@!=4.5pc{K_n(F)\ar[r]^{q}\ar[d]_{\pi_n} & k_n(F) \\
  K_n(F/_m\dot F^2)\ar[ur]_{\tilde q_n}}$$
  and then, $\tilde q_n$ is surjective. Finally, if $\tilde q_n(\rho([a_1])...\rho([a_n]))=0$, then $\tilde\rho(a_1)...\tilde\rho(a_n)=0$, and hence $\rho(a_1)...\rho(a_n)\in2K_n(F)$. By Lemma \ref{lemktmulti1}
  $$\rho(a_1)...\rho(a_n)=\sum^p_{j=1}\rho(d_{j1})...\rho(d_{jn}),\,d_{ij}\in\dot F$$
  and for all $i$ there is an index $k$ such that $a_{ik}=b_i^2$, $b_i\in\dot F$.
  Therefore
  \begin{align*}
     \pi_n(\rho(a_1)...\rho(a_n))&=\pi_n\left(\sum^p_{j=1}\rho(d_{j1})...\rho(d_{jn})\right) \\
     &=\sum^p_{j=1}\pi_n(\rho(d_{j1})...\rho(d_{jn}))
     =\sum^p_{j=1}\rho([d_{j1}])...\rho([d_{jn})]) \\
     &=\sum^p_{j=1}[d_{j1}])...\rho([1])...\rho([d_{jn})]=0.
  \end{align*}
  Then $\mbox{Ker}(\tilde q_n)=[0]$, proving that $\tilde q_n$ is injective. Then $\tilde q_n$ is an isomorphism, and composing all the isomorphisms obtained here we get
  $$k(F)\cong K(F/_m \dot F^2)\cong k(F/_m\dot F^2).$$
  \end{proof}

The Theorem \ref{fixsg3} below generalizes Proposition 5.10 of \cite{wadsworth55merkurjev}: this  constitutes a fundamental technical step to build profinite (Galois) groups associated to a pre-special hyperfield in \cite{roberto2022ACmultifields2}.

Lets establish some notation: for $n\ge0$ we denote
 $$P(n)=\mathcal P(\{0,...,n-1\})\setminus\{\emptyset\}$$
 and for $0\le i\le n-1$, denote
 $$P(n,i)=\{X\in P(n):i\in X\}.$$
 For a be a pre-special hyperfield $F$ and $\{a_0,...,a_{n-1}\}\subseteq F^*$ $\mathbb F_2$-linearly independent, if $S\in P(n)$ we denote
 $$a_S:=a_0^{\varepsilon_0}...a_{n-1}^{\varepsilon_{n-1}},$$
 where $\varepsilon_0\in\{0,1\}$ for all $i=0,..,n-1$ and $\varepsilon_i=1$ if and only if $i\in S$.

 Remember that by the very definition of $k_n(F)$,
 $$k_n(F):=[k_1(G)\otimes k_1(G)]/M,$$
 where $M$ is the subgroup of $k_1(G)\otimes k_1(G)$ generated by
 $$\{\rho(a)\rho(b):a\in D_G(1,b),\,a\}.$$

 \begin{teo}\label{fixsg3}
Let $F$ be a pre-special hyperfield and $\{a_0,...,a_{n-1}\}\subseteq F^*$ $\mathbb F_2$-linearly independent. The following conditions are equivalent:
\begin{enumerate}[i -]
    \item There exists $\{b_0,...,b_{n-1}\}\subseteq F^*$ such that
    $$\sum_{k<n}\rho(a_k)\rho(b_k)=0\mbox{ in }k_2(F).$$

    \item There exist subsets $\{c_0,...,c_{m-1}\},\{d_0,...,d_{n-1}\}$ of $F^*$ with $m\ge n$ such that
    \begin{enumerate}
        \item $\{c_0,...,c_{m-1}\}$ is linearly independent and $c_i=a_i$ for all $i<n$;
        \item $d_i=b_i$ for all $i<n$ and $d_i=1$ for $i=n,...,m-1$.
        \item For all $x\in C:=[c_0,...,c_{m-1}]$, there is some $r_x\in 1+x\setminus\{0\}$ such that for each $i<m$
        $$d_i=\prod_{x\in C_i}r_x$$
        where
        $$C_i=\left\lbrace\prod_{k<m}c_k^{\varepsilon_k}:\varepsilon_k\in\{0,1\}\mbox{ and }\varepsilon_i=1\right\rbrace.$$
    \end{enumerate}
\end{enumerate}
\end{teo}

In other words, $C_i$ is "counting" all products $c_0^{r_0}...c_i^1...c_{m-1}^{r_{m-1}}$. Since for all $x\in C:=[c_0,...,c_{m-1}]$ there exist $S\in P(m)$ such that
$$x=\prod_{i\in S}c_i:=c_S.$$
Denoting $r_x$ by $r_S$ we can rewrite
$$d_i=\prod_{x\in C_i}r_x=\prod_{S\in P(m)}r_S.$$

\begin{proof}[Proof of Theorem \ref{fixsg3}]

 (i) $\Rightarrow$ (ii). Let$$\sum_{k<n}\rho(a_k)\rho(b_k)=0\mbox{ in }k_2(F).$$
 Then there exist $u_0,...,u_{p-1},v_0,...,v_{p-1}\in F^*$ such that $v_i\in 1+u_i$ for $i=0,...,p-1$ and
 $$\sum_{k<n}\rho(a_k)\rho(b_k)=\sum_{k<n}\rho(a_k)\rho(b_k)\mbox{ in }k_1(F)\otimes k_1(F).$$
 Enlarge the set $\{a_0,...,a_{n-1}\}$ to a base $\{c_0,...,c_{m-1}\}$ of $[\{a_0,...,a_{n-1},u_0,...,u_{p-1}\}]$, with $c_i=a_i$ for all $i<n$. For all $x\in C:=[c_0,...,c_{m-1}]$ there exist $S\in P(m)$ such that
 $$x=\prod_{i\in S}c_i:=c_S.$$
 Moreover, since $\{c_0,...,c_{m-1}\}$ is a basis, for each $i=0,...,p-1$ there is only one $S_i\in P(m)$ such that
 $$u_i=c_{S_i}.$$
 For each $S\in P(m)$, set
 $$r_S:=\prod_{\substack{those\,j\,with\\ S_j=S}}v_j.$$
 If no $S_j=S$, set $r_S=1$. Note that if there is an index $j$ with $S=S_j$, this index must be unique (because the expression $u_i=c_{S_i}$ is unique). Then by construction $r_S\in1+c_S\setminus\{0\}$ and in $k_2(F)$ we get
 \begin{align*}
     \sum_{k<m}\rho(a_k)\rho(b_k)&=\sum_{k<n}\rho(c_k)\rho(d_k)=\sum_{k<p}\rho(u_k)\rho(v_k) \\
     &=\sum_{S\in P(m)}\rho\left(\prod_{\substack{those\,j\,with\\S_j=S}}
     c_j\right)\rho(v_j) \\
     &=\sum_{S\in P(m)}\sum_{\substack{those\,j\,with\\S_j=S}}
     \rho(c_j)\rho(v_j) \\
     &=\sum_{S\in P(m)}\rho(c_S)\rho\left(
     \prod_{\substack{those\,j\,with\\S_j=S}}v_S\right) \\
     &=\sum_{S\in P(m)}\rho(c_S)\rho(r_S)=
     \sum_{S\in P(m)}\sum_{k\in S}\rho(c_k)\rho(r_S) \\
     &=\sum_{k<m}\rho(c_k)\rho\left(\prod_{S\in P(n)}r_S\right).
 \end{align*}
 Since $\{c_0,...,c_{m-1}\}$ is a basis, it follows that
 $$d_i=\prod_{S\in P(n)}r_S$$
 as desired.

 (ii)$\Rightarrow$(i). Under the hypotheses of (ii) we get
 \begin{align*}
     \sum_{k<n}\rho(a_k)\rho(b_k)&=\sum_{k<m}\rho(c_k)\rho(d_k)
     =\sum_{k<m}\rho(c_k)\rho\left(\prod_{S\in P(n)}r_S\right) \\
     &=\sum_{k<m}\sum_{S\in P(m)}\rho(c_k)\rho(r_S)
     =\sum_{S\in P(m)}\sum_{k<m}\rho(c_i)\rho(r_S) \\
     &=\sum_{S\in P(m)}\rho(c_S)\rho(r_S)=0.
 \end{align*}
\end{proof}

  \section{Inductive Graded Rings and Interchanging of K-theories}

  After the three K-theories defined in the above sections, it is desirable (or, at least, suggestive) to build of an abstract environment that encapsulates all them, and of course, provide an axiomatic approach to guide new extensions of the concept of K-theory in the context of the algebraic and abstract theories of quadratic forms. The inductive graded rings fit well to this purpose\footnote{A categorical and systematic development of inductive graded rings in connection with "quadratic multirings" is carried out in \cite{roberto2022graded}.}.

\begin{defn}[Inductive Graded Rings First Version (Definition 9.7 of \cite{dickmann2000special})]\label{igr1}
 An \textbf{inductive graded ring} (or \textbf{Igr} for short) is a structure $R=((R_n)_{n\ge0},(h_n)_{n\ge0},\ast_{nm})$ where
\begin{enumerate}[i -]
    \item $R_0\cong\mathbb F_2$.
    \item $R_n$ is a group of exponent 2 with a distinguished element $\top_n$.
    \item $h_n:R_n\rightarrow R_{n+1}$ is a group homomorphism such that $h_n(\top_n)=\top_{n+1}$.
    \item For all $n\ge0$, $h_n=\ast_{1n}(\top_1,\_)$.
    \item The ring
    $$R=\bigoplus_{n\ge0}R_n$$
    is a commutative graded ring.
    \item For $0\le s\le t$ define
    $$h^t_s=\begin{cases}Id_{R_s}\mbox{ if }s=t\\
    h_{t-1}\circ...\circ h_{s+1}\circ h_s\mbox{ if }s<t.\end{cases}$$
    Then if $p\ge n$ and $q\ge m$, for all $x\in R_n$ and $y\in R_m$,
    $$h^p_n(x)\ast h^q_m(y)=h^{p+q}_{n+m}(x\ast y).$$
\end{enumerate}
 A \textbf{morphism} between Igr's $R$ and $S$ is a pair $f=(f,(f_n)_{n\ge0})$ where $f_n:R_n\rightarrow S_n$ is a
morphism of pointed groups and
$$f=\bigoplus\limits_{n\ge0}f_n:R\rightarrow S$$
is a morphism of commutative rings with unity. The category of inductive graded rings (in first version) and their morphisms will be denoted by $\mbox{Igr}$.
\end{defn}

A first consequence of these definitions is that: if
$$f:((R_n)_{n\ge0},(h_n)_{n\ge0},\ast_{nm})\rightarrow ((S_n)_{n\ge0},(l_n)_{n\ge0},\ast_{nm})$$
is a morphism of Igr's then $f_{n+1}\circ h_n=l_n\circ f_n$.
$$\xymatrix@!=2.5pc{R_0\ar[r]^{h_0}\ar[d]_{f_0} & R_1\ar[r]^{h_1}\ar[d]_{f_1} & R_2\ar[r]^{h_2}\ar[d]_{f_2} &
...\ar[r]^{h_{n-1}} & R_n\ar[r]^{h_n}\ar[d]_{f_n} & R_{n+1}\ar[r]^{h_{n+1}}\ar[d]_{f_{n+1}} & ... \\
S_0\ar[r]^{l_0} & S_1\ar[r]^{l_1} & S_2\ar[r]^{l_2} & ...\ar[r]^{l_{n-1}} & S_n\ar[r]^{l_n} & S_{n+1}\ar[r]^{l_{n+1}} &
...}$$
In fact, $f_0:R_0\rightarrow S_0$ is an isomorphism, so $f_1\circ h_0=l_0\circ f_0$. If $n\ge1$, for all $a_n\in R_n$
holds
\begin{align*}
 f_{n+1}\circ h_n(a_n)&=f_{n+1}\circ(\ast_{1n}(\top_1,a_n))=f_1(\top_1)\ast_{1n}f_n(a_n) \\
 &=\top_1\ast_{1n}f_n(a_n)=l_n(f_n(a_n))=l_n\circ f_n(a_n).
\end{align*}

\begin{ex}\label{ex1}
$ $
 \begin{enumerate}[a -]
  \item  Let $F$ be a field of characteristic not 2. The main actors here are $WF$, the Witt ring of $F$ and $IF$, the fundamental ideal of $WF$. Is well know that $I^nF$, the $n$-th power of $IF$ is additively generated by $n$-fold Pfister forms over $F$. Now, let $R_0=WF/IF\cong\mathbb F_2$
and $R_n=I^nF/I^{n+1}F$. Finally, let $h_n=\_\otimes\langle1,1\rangle$. With these prescriptions we have an inductive graded ring $R$ associated to $F$.

  \item The previous example still works if we change the Witt ring of a field $F$ for the Witt ring of a (formally real) special group $G$.
 \end{enumerate}
\end{ex}

\begin{defn}
 We denote $\mbox{Igr}_{fin}$ the full subcategory of $\mbox{Igr}$ such that
 $$\mbox{Obj}(\mbox{Igr}_{fin})=\left\lbrace R\in\mbox{Obj}(\mbox{Igr}):
 \left| R_n\right|<\omega\mbox{ for all
}n\ge1\right\rbrace.$$
\end{defn}

Of course,
$$\left\lbrace R\in\mbox{Obj}(\mbox{Igr}):\left|\bigoplus\limits_{n\ge1}
R_n\right|<\omega\right\rbrace\ne\mbox{Obj}(\mbox{Igr}_{fin}),$$
for example, in \ref{ex1}(a), $\mathbb F_2[x]\in\mbox{Obj}(\mbox{Igr}_{fin})$ and $\mathbb F_2[x]$ is not finite.

  We finally this section with an use of Igr's to interchanging the three K-theory notions presented before in a functorial fashion. Lets
first, look more carefully at theorem \ref{km1}. We make the following distinctions between K-theories:
\begin{align*}
 K^{mil}&\mbox{ will denote the Milnor's K-theory}, \\
 K^{dm}&\mbox{ will denote the Dickmann-Miraglia's K-theory}, \\
 K^{mult}&\mbox{ will denote the K-theory of Hyperfields}.
\end{align*}

We will deal with some different categories, so in order to ease the exposition lets establish some dictionary:

\begin{enumerate}
    \item An object in $Igr$ will be denoted simply by $R=(R_n,h_n,\ast_{nm})$ (being implicit the identation $n\in\mathbb N$ or $n\ge0$). We even denote simply $R=(R_n,h,\ast)$ if the context allows it.

    \item The distinguished elements $\top_n\in R_n$ will be denoted simply by $\top$, whenever the context allows it, and we do the same for the symbols $+_n,0_n$.

    \item An IGR-morphism between $R$ and $S$ will be simply denoted by $f:R\rightarrow S$, with the immediate convention that $f=(f_n)_{n\ge0}$, with $f_n:R_n\rightarrow S_n$ being a morphism of pointed $\mathbb F_2$-modules.

    \item Let $\mathcal C$ be a category. We use the notation "$A,B\in\mathcal C$" as a synonym of "$A,B$ be objects of $\mathcal C$".

    \item Let $\mathcal C$ be a category, $A,B$ be objects of $\mathcal C$. We use the notations "$f\in\mathcal C(A,B)$" and $f\in\mbox{Hom}_{\mathcal C}(A,B)$" as a synonyms of "$f:A\rightarrow B$ is a morphism in $\mathcal C$".

    \item We will define various functors below. In order to facilitate the exposition, we define a functor $F:\mathcal C\rightarrow\mathcal D$ by the sentence "for $f:A\rightarrow B$, we define $F(A):=...$ and $F(f):=...$"; being implicit that $A$ and $B$ are objects and $f:A\rightarrow B$ is a morphism  in category $\mathcal C$.
\end{enumerate}

 Of course, we need the following theorem:
\begin{teo}\label{km1}
$ $
 \begin{enumerate}[a -]
  \item Let $F$ be a field. Then $k^{mil}_*F$ (the reduced Milnor K-theory) is an inductive graded ring.
  \item Let $G$ be a special group. Then $k^{dm}_*G$ (the Dickmann-Miraglia K-theory of $G$) is an inductive graded ring.
  \item Let $F$ be a hyperfield. Then $k^{mult}_*F$ (our reduced K-theory) is an inductive graded ring.
 \end{enumerate}
\end{teo}
\begin{proof}
Item (a) is the content of Lemma 9.11 in \cite{dickmann2000special}, and item (b) is the content of Lemma 9.12 in \cite{dickmann2000special}. We prove item (c) and itens (a) and (b) will proceed by the same argument.

Let $k^{mult}_*F=(k_0F,k_1F,...,k_nF,...)$ be the reduced K-theory of a hyperfield $F$. Let $\top_0=1$ and for each $n\ge1$, we set $\top_n=l(-1)^n$ as the distinguished element of $m-n$. For each $n\ge0$, let
$\theta_n:\prod^n_{j=1}K^{mult}_1F\rightarrow\otimes^{n+1}_{j=1}K^{mult}_{n+1}F$ given by the rule
$$\theta_n(\rho(a_1),...,\rho(a_n)):=\rho(-1)\rho(a_1)...\rho(a_n).$$
We have for each $i\in\{1,...,n\}$ and each $a_1,...,a_n,b_i\in F^*$ that
\begin{align*}
    &\theta_n(\rho(a_1),...\rho(a_i)+\rho(b_i),...,\rho(a_n))=
    \theta_n(\rho(a_1),...\rho(a_ib_i),...,\rho(a_n)):= \\
    &\rho(-1)\rho(a_1)...\rho(a_ib_i)...\rho(a_n)
    =\rho(-1)\rho(a_1)...[\rho(a_i)+\rho(b_i)]...\rho(a_n)= \\
    &\rho(-1)\rho(a_1)...\rho(a_i)...\rho(a_n)+
    \rho(-1)\rho(a_1)...\rho(b_i)...\rho(a_n)= \\
    &\theta_n(\rho(a_1),...,\rho(a_i),...,\rho(a_n))+
    \theta_n(\rho(a_1)...\rho(b_i)...\rho(a_n)),
\end{align*}
then $\theta_n$ is multilinear. By the universal property of tensor product, we have a group homomorphism $\tilde\theta_n:K^{mult}_nF\rightarrow K^{mult}_{n+1}F$ given by the rule\footnote{Remember that we are using the simplificated notation for elements in $K^{mult}_nF$ (and all other K-theories), which is
$\rho(a_1)...\rho(a_n):=\rho(a_1)\otimes...\otimes\rho(a_n)$.}
$$\tilde\theta_n(\rho(a_1)...\rho(a_n))=\rho(-1)\rho(a_1)...\rho(a_n).$$
In order to make distinctions between reduced and non-reduced K-theories, we pontually denote an element in $k^{mult}_nF:=K^{mult}_nF/2K^{mult}_nF$ by $\tilde\rho(a_1)...\tilde\rho(a_n)$. Lets also denote the canonical projection by $\pi_n:K^{mult}_nF\rightarrow k^{mult}_nF$. We define $\omega_n:k^{mult}_nF\rightarrow k^{mult}_{n+1}F$ by the following rule (on generators): for $a_1,...,a_n\in\dot F$,
    $$\omega_n(\tilde\rho(a_1)...\tilde\rho(a_n))=\tilde\rho(-1)\tilde\rho(a_1)...\tilde\rho(a_n).$$
In fact, if $\rho(a_1)...\rho(b_n)-\rho(b_1)...\rho(b_n)\in 2K^{mult}_nF$ then
\begin{align*}
    \rho(-1)\rho(a_1)...\rho(b_n)-\rho(-1)\rho(b_1)...\rho(b_n)
    =\rho(-1)[\rho(a_1)...\rho(b_n)-\rho(b_1)...\rho(b_n)]\in 2K^{mult}_{n+1}F,
\end{align*}
which proves that $\omega_n$ is in fact a group homomorphism making the following diagram commute
$$\xymatrix@!=4.pc{K^{mult}_nF\ar[r]^{\tilde\theta_n}\ar[d]_{\pi_n} & K^{mult}_{n+1}F\ar[d]^{\pi_{n+1}} \\
k^{mult}_nF\ar[r]_{\omega_n} & k^{mult}_{n+1}F}$$
With these rules, we already have the properties (i)-(iv) of Definition \ref{igr1} holding in $k_*F$, remaining only property (v). Note that $\omega^t_s=\tilde\rho(-1)^{t-s}$ for $0\le s<t$.

    Now let $m,n,p,q\in\mathbb N$, $p\ge n$, $q\ge m$ and consider $x\in k^{mult}_nF$ and $y\in k^{mult}_mF$. Note that $\omega^p_n(x)=\tilde\rho(-1)^{p-n}\cdot x$ and $\omega^q_m(y)=\tilde\rho(-1)^{q-m}\cdot y$. Then
    $$\omega^p_n(x)\cdot\omega^q_m(y)=(\tilde\rho(-1)^{p-n}\cdot x)(\tilde\rho(-1)^{q-m}\cdot y)=\tilde\rho(-1)^{p-n+q-m}\cdot (x\cdot y)=\omega^{p+q}_{n-m}(x),$$
    completing the proof.
\end{proof}

Using this Theorem (in addition with the argument of Lemma 3.3 in \cite{dickmann2006algebraic}) we obtain the following.

  \begin{cor}\label{km2}
   Let $F$ be a field. We have a functor and $k:Field_2\rightarrow \mbox{Igr}$ induced by K-theory and Milnor's reduced K-theory.
  \end{cor}

  \begin{cor}\label{km3}
   Let $F$ be a hyperfield. We have a functor $k^{mult}:MField_2\rightarrow \mbox{Igr}$ induced by our reduced K-theory.
  \end{cor}

  \begin{teo}[Theorem 2.5 in \cite{dickmann2003lam}]\label{km4}
   Let $F$ be a field. The functor $G:Field_2\rightarrow SG$ provides a functor $k^{dm}_*:Field_2\rightarrow \mbox{Igr}$ (the special group
K-theory functor) given on the objects by $k^{dm}_*(F):k^{dm}_*(G(F))$ and on the morphisms $f:F\rightarrow K$ by $k^{dm}_*(f):=G(f)_*$ (in the sense of Lemma 3.3 of \cite{dickmann2006algebraic}). Moreover, this functor commutes with the functors $G$ and $k$, i.e, for all $F\in Field$,
$k^{dm}_*(G(F))\cong k_*(F)$.
  \end{teo}

  \begin{teo}\label{km5}
   Let $G$ be a special group. The equivalence of categories $M:SG\rightarrow SMF$ induces a functor $k^{mult}_*:SG\rightarrow
\mbox{Igr}$ given on the objects by $k^{mult}_*(G):=k^{mult}_*(M(G))$ and on the morphisms $f:G\rightarrow H$ by
$k^{mult}_*(f):=k^{mult}_*(M(f))$. Moreover, this functor commutes with $M$ and $k^{dm}$, i.e, for all $G\in SG$, $k^{mult}_*(M(G))\cong k^{dm}_*(G)$.
  \end{teo}
  \begin{proof}
   The only part requiring proof is that for all $G\in SG$, $k^{mult}_*(M(G))\cong k^{dm}_*(G)$. The very first observation is that: since $G$ is an exponent 2 group, the reduced and non-reduced $K^{mult}$-theory of $M(G)$ coincide.

   Following the argument of Theorem 2.5 in \cite{dickmann2003lam}, it is enough to show the following two statements:
   \begin{enumerate}[i -]
       \item For all $a,b\in G$, if $b\in1-a$ in $M(G)$ then $\lambda(b)\lambda(a)=0$;
       \item For all $a,b\in G$, if $b\in D_G(1,a)$ then $\rho(b)\rho(a)=0$.
   \end{enumerate}

   For (i), if $b\in1-a$ in $M(G)$ then $b\in D_G(1,-a)$ and then, $\lambda(b)\lambda(-a)=0$. Hence
   $$\lambda(b)^2=\lambda(b)\lambda(-a)=\lambda(b)\lambda(a)+\lambda(b)\lambda(-1).$$
   Since $\lambda(b)\lambda(-1)=\lambda(b)^2$, we get $\lambda(b)\lambda(a)=0$.

   For (ii) we just use the same argument: if $b\in D_G(1,a)$ then $b\in1+a$ in $M(G)$ and then, $\rho(b)\rho(-a)=0$. Hence
   $$\rho(b)^2=\rho(b)\rho(-a)=\rho(b)\rho(a)+\rho(b)\rho(-1).$$
   Since $\rho(b)\rho(-1)=\rho(b)^2$, we get $\rho(b)\rho(a)=0$.
  \end{proof}

  Combining Theorems \ref{km1}, \ref{km4}, \ref{km5}, \ref{ktmarshall2} and Corollaries \ref{km2} and \ref{km3} we obtain the following Theorem, that unificate in some sense all three K-theories:

  \begin{teo}[Interchanging K-theories Formulas]\label{res0}
   Let $F\in Field_2$. Then
   $$k^{mil}(F)\cong k^{dm}(G(F))\cong k^{mult}(M(G(F))).$$
   If $F$ is formally real and $T$ is a preordering of $F$, then
   $$k^{dm}(G_T(F))\cong k^{mult}(M(G_T(F))).$$
   Moreover, since $M(G(F))\cong F/_m\dot F^2$ and $M(G_T(F))\cong F/_mT^*$, we get
   \begin{align*}
       k^{mil}(F)&\cong k^{dm}(G(F))\cong k^{mult}(F/_m\dot F^2)\mbox{ and } \\
       k^{dm}(G_T(F))&\cong k^{mult}(F/_mT^*).
   \end{align*}
  \end{teo}

  \begin{cor}\label{res1}
  Let $F$ be a field. Then
  $$k^{mil}(F)\cong k^{mult}(F/_m\dot F^2).$$
  \end{cor}
  \begin{proof}
  Using the previous Corollary, we already have
  $$k^{mil}(F)\cong k^{dm}(G(F))\cong k^{mult}(M(G(F))).$$
  Now, is enough to observe that $M(G(F))\cong F/_m\dot F^2$.
  \end{proof}

  Combining Theorem \ref{res0}, Corollary \ref{res1} and Theorem \ref{ktmarshall2} we get the following Corollaries.

  \begin{cor}\label{res2}
   Let $F$ be a formally real field and $T$ be a preordering. Then we have a surjective map
   $$k^{mil}(F)\rightarrow k^{mult}(F/_m T^*).$$
  \end{cor}

  \begin{cor}\label{res3}
  Let $G$ be a pre-special group and $H\subseteq G$ be a subgroup of $G$. Let $M(G)$ be the pre-special multifield associated to $G$ and $M(H)=H\cup\{0\}\subseteq M(G)$. Then
  $$G/H\cong M(G)/_mM(H)^*.$$
  Moreover, $M(H)\subseteq M(G)$ is a preordering if and only if $H$ is saturated.
  \end{cor}

  \begin{cor}\label{res4}
   Let $G$ be a special group and $H$ be a saturated subgroup. Then we have a surjective map
   $$k^{dm}(G)\rightarrow k^{mult}(G/_m H)\cong k^{dm}(G/H).$$
  \end{cor}

 We close this paper "embedding" the three K-theories considered until now into the framework of inductive graded rings, in such a way that this "embedding" provides a sort of free objects in $\mbox{Igr}$. Lets firstly define functors that formalize the construction of K-theory.

 \begin{defn}\label{level1}
 We define the \textbf{graded subring generated by the level 1} functor
 $$\mathbbm1:\mbox{Igr}\rightarrow\mbox{Igr}$$
 as follow: for an object $R=(R_n,h_n,\ast_{nm})$,
\begin{enumerate}[i -]
 \item $\mathbbm1(R)_0:=R_0\cong\mathbb F_2$,
 \item $\mathbbm1(R)_1:=R_1$,
 \item for $n\ge2$,
 \begin{align*}
  \mathbbm1(R)_n&:=\{x\in R_n:x=\sum\limits^{r}_{j=1}a_{1j}\ast_{11}...\ast_{11}a_{nj}, \\
  &\mbox{with }a_{ij}\in R_1,\,1\le i\le n,\,1\le j\le r\mbox{ for some }r\ge1\}.
 \end{align*}
 Of course, $\mathbbm1(R)$ provides an inclusion $\iota_{\mathbbm1(R)}:\mathbbm1(R)\rightarrow R$ in the obvious way.
\end{enumerate}

On the morphisms, for $f\in\mbox{Igr}(R,S)$, we define $\mathbbm1(f)\in\mbox{Igr}(\mathbbm1(R),\mathbbm1(S))$ by
the restriction $\mathbbm1(f)=f\upharpoonleft_{\mathbbm1(R)}$. In other words, $\mathbbm1(f)$ is the only Igr-morphisms that
makes the following diagram commute:
$$\xymatrix@!=4.5pc{\mathbbm1(R)\ar[r]^{\iota_{\mathbbm1(R)}}\ar[d]_{\mathbbm1(f)} & R\ar[d]^{f} \\
\mathbbm1(S)\ar[r]_{\iota_{\mathbbm1(S)}} & S}$$

We denote $\mbox{Igr}_{\mathbbm1}$ the full subcategory of $\mbox{Igr}$ such that
$$\mbox{Obj}(\mbox{Igr}_{\mathbbm1})=\{R\in\mbox{Igr}:\iota_{\mathbbm1(R)}:\mathbbm1(R)\rightarrow R\mbox{ is an isomorphism}\}.$$
\end{defn}

\begin{defn}\label{quotop}
 We define the \textbf{quotient graded ring functor}
 $$\mathcal Q:\mbox{Igr}\rightarrow\mbox{Igr}$$
 as follow: for
an object $R=(R_n,h_n,\ast_{nm})$, $\mathcal Q(R):=R/Q$, where $Q=(Q_n)_{n\ge0}$ is the ideal generated by
$\{(\top+a)\ast_{11}a\in R_2:a\in R_1\}$. More explicit,
\begin{enumerate}[i -]
 \item $Q_0:=\{0\}\subseteq R_0$,
 \item $Q_1:=\{0\}\subseteq R_1$,
 \item for $n\ge2$, $Q_n\subseteq R_n$ is the pointed $\mathbb F_2$-submodule generated by
 \begin{align*}
  \{x\in R_n:x&=y_l\ast_{l1}(\top+a_1)\ast_{11}a_1\ast_{1r}z_r, \\
  &\mbox{with }a_1\in R_1,\,y_l\in R_l,\,z_r\in R_r,\,1\le r,l\le n-2,\,l+r=n-2\}.
 \end{align*}
 Of course, $\mathcal Q(R)$ provides a projection $\pi_{\mathcal Q(R)}:R\rightarrow\mathcal Q(R)$ in the obvious way.
\end{enumerate}

On the morphisms, for $f\in\mbox{Igr}(R,S)$, we define $\mathcal Q(f)\in\mbox{Igr}(\mathcal Q(R),\mathcal Q(S))$ by
the only Igr-morphisms that makes the following diagram commute:
$$\xymatrix@!=4.5pc{R\ar[r]^{\pi_{\mathcal Q(R)}}\ar[d]_{f} & \mathcal Q(R)\ar[d]^{\mathcal Q(f)} \\
S\ar[r]_{\pi_{\mathcal Q(S)}} & \mathcal Q(S)}$$

We denote $\mbox{Igr}_q$ the full subcategory of $\mbox{Igr}$ such that
$$\mbox{Obj}(\mbox{Igr}_q)=\{R\in\mbox{Igr}:\pi_{\mathcal Q}:R\rightarrow\mathcal Q(R)\mbox{ is an isomorphism}\}.$$
\end{defn}

\begin{defn}
 We denote by $\mbox{Igr}_+$ the full subcategory of $\mbox{Igr}$ such that
 $$\mbox{Obj}(\mbox{Igr}_+)=\mbox{Obj}(\mbox{Igr}_{\mathbbm1})\cap\mbox{Obj}(\mbox{Igr}_q).$$
 We denote by $j_+:\mbox{Igr}_+\rightarrow\mbox{Igr}$ the inclusion functor.
\end{defn}

By the very definition of the K-theory of hyperfields (with the notations in Theorem \ref{3.3ktmultiadap}) we define the following functor.

 \begin{defn}[K-theories Functors]
  With the notations of Theorem \ref{3.3ktmultiadap} we have a functors $k:\mathcal{HMF}\rightarrow\mbox{Igr}_+$, $k:\mathcal{PSMF}\rightarrow\mbox{Igr}_+$ induced by the reduced K-theory for hyperfields.
 \end{defn}

 Now, let $R\in\mbox{Igr}$. We define a hyperfield $(\Gamma(R),+,-.\cdot,0,1)$ by the following: firstly, fix an isomorphism $e_R:(R_1,+_1,0_1,\top_1)\rightarrow (G(R),\cdot,1,-1)$. This isomorphism makes, for example, an element $a\ast_{11}(\top+b)\in R_2$, $a,b\in R_1$ take the form $(e^{-1}_R(x))\ast_{11}(e^{-1}_R((-1)\cdot y))\in R_2$, $x,y\in\Gamma(R)$. By an abuse of notation, we simply write $x\ast_{11}(-y)\in R_2$, $x,y\in\Gamma(R)$. In this sense, an element in $Q_2$ (see Definition \ref{quotop}) has the form $x\ast_{11}(-x)$, $x\in \Gamma(R)$, and we can extend this terminology for all $Q_n$, $n\ge2$.

 Now, let $\Gamma(R):=G(R)\cup\{0\}$ and for $a,b\in\Gamma(R)$ we define
 \begin{align}\label{gammahyper}
    -a&:=(-1)\cdot a,\nonumber\\
    a\cdot0&=0\cdot a:=0,\nonumber\\
    a+0&=0+a=\{a\},\nonumber \\
    a+(-a)&=\Gamma(R),\nonumber \\
    \mbox{for }&a,b\ne0,a\ne-b\mbox{ define }\nonumber\\
    a+b&:=\{c\in\Gamma(R):\mbox{there exist }d\in G(R)\mbox{ such that } \nonumber\\
    &a\cdot b=c\cdot d\mbox{ and }a\ast_{11}b=c\ast_{11}d\in R_2\}.
\end{align}

\begin{prop}\label{prespechf}
With the above rules, $(\Gamma(R),+,-.\cdot,0,1)$ is a pre-special hyperfield.
\end{prop}
\begin{proof}
The proof is similar to Theorem \ref{hfproduct}: we will verify the conditions of Definition \ref{defn:multiring}. Note that by the definition of multivalued sum once we proof that $\Gamma(R)$ is an hyperfield, it will be hyperbolic.

\begin{enumerate}[i -]
 \item In order to prove that $(\Gamma(R),+,-.\cdot,0,1)$ is a multigroup we follow the steps below.
 \begin{enumerate}[a -]
   \item Commutativity and $(a\in b+0)\Leftrightarrow(a=b)$ are direct consequence of the definition of multivaluated sum and the fact that $a\ast_{11}b=b\ast_{11}a$.

  \item We will prove that if $c\in a+b$, then $a\in b-c$ and $b\in -c+a$.

  If $a=0$ (or $b=0$) or $a=-b$, then $c\in a+b$ means $c=a$ or $c\in a-a$. In both cases we get $a\in b-c$ and $b\in -c+a$.

  Now suppose $a,b\ne0$ with $a\ne-b$. Let $c\in a+b$. Then $a\cdot b=c\cdot d$ and $a\ast_{11}b=c\ast_{11}d\in R_2$ for some $d\in G(R)$. Then $b\cdot(-c)=a\cdot(-d)$ and
  \begin{align*}
      a\ast_{11}b=c\ast_{11}d&\Rightarrow
      a\ast_{11}b\ast_{11}[(-b)\ast_{11}(-d)]=c\ast_{11}d\ast_{11}[(-b)\ast_{11}(-d)] \\
      &\Rightarrow a\ast_{11}(-d)=b\ast_{11}(-c),
  \end{align*}
  proving that $a\in b-c$. Similarly we prove that $b\in -c+a$.

  \item Now we prove the associativity, that is,
  $$(a+b)+c=a+(b+c).$$
  In fact (see the remarks after Lemma 2.4 of \cite{ribeiro2016functorial}), it is enough to show
  $$(a+b)+c\subseteq a+(b+c).$$

  If $0\in\{a,b,c\}$ we are done. Now let $a,b,c\ne0$ and
  $x\in(a+b)+c$. If $b=-c$ we have
  $$a+(b+c)=a+\Gamma(R)=\Gamma(R)\supseteq(a+b)+c.$$
  If $-c\in a+b$, then $-a\in b+c$ and we have
  $$(a+b)+c=\Gamma(R)=a+(b+c).$$

  Now suppose $a,b,c\ne0$, $c\ne-b$, $-c\notin a+b$. Let
  $p\in (a+b)+c$. Then $p\in x+c$ for some $x\in a+b$, and there exists $q,y\in G(R)$ such that
  \begin{align*}
      p\cdot q=x\cdot c&\mbox{ and }p\ast_{11}q=x\ast_{11}c \\
      x\cdot y=a\cdot b&\mbox{ and }x\ast_{11}y=a\ast_{11}b.
  \end{align*}
  Now let $z,w\in G(R)$ defined by $z=xca$, $w=abx$. We have $p\cdot q=x\cdot c=x\cdot c\cdot(a^2)=a\cdot z$ and $z\cdot w=(xca)\cdot(abx)=b\cdot c$. Moreover
  \begin{align*}
      p\ast_{11}q&=x\ast_{11}c=x\ast_{11}c\ast_{11}(a^2)=x\ast_{11}c\ast_{11}a\ast_{11}a \\
      &=[x\ast_{11}c\ast_{11}a]\ast_{11}a=(xca)\ast_{11}a=z\ast_{11}a
  \end{align*}
  and
  \begin{align*}
      z\ast_{11}w&=(xca)\ast_{11}(abx)=(cxa)\ast_{11}(axb)=
      c\ast_{11}(xa)\ast_{11}(ax)\ast_{11}b=c\ast_{11}b.
  \end{align*}
  Then $p\in a+z$ with $z\in c+b=b+c$, and hence $p\in a+(b+c)$.
 \end{enumerate}

   \item Since $(G(R),\cdot,1)$ is an abelian group, we conclude that $(\Gamma(R),\cdot,1)$ is a commutative monoid. Beyond this, every nonzero element $a\in\Gamma(R)$ is such that $a^2=1$.

  \item $a\cdot0=0$ for all $a\in\Gamma(R)$ is direct from definition.

  \item For the distributive property, let $a,b,d\in\Gamma(R)$ and consider $x\in d(a+b)$. We need to prove that
  \begin{align*}
      \tag{*}x\in d\cdot a+d\cdot b.
  \end{align*}
  It is the case if $0\in\{a,b,d\}$ or if $b=-a$.

  Now suppose $a,b,d\ne0$ with $b\ne-a$. Then there exist $y\in G(R)$ such that $x=dy$ and $y\in a+b$. Moreover, there exist some $z\in G(R)$ such that $y\cdot z=a\cdot b$ and $y\ast_{11}z=a\ast_{11}b$. Therefore $(dy)\cdot(dz)=(da)\cdot(db)$ and $(dy)\ast_{11}(dz)=(da)\ast_{11}(db)$, and $x=dy\in d\cdot a+d\cdot b$.
\end{enumerate}
Then $(\Gamma(R),+,-,\cdot,0,1)$ is a hyperbolic hyperfield.

Finally, let $a\in\Gamma(R)$ and $x,y\in1-a$. If $a=0$ or $a=1$ then we automatically have $x\cdot y\in 1-a$, so let $a\ne0$ and $a\ne1$. Then $x,y\in G(R)$ and there exist $p,q\in\Gamma(R)$ such that
  \begin{align*}
      x\cdot p=1\cdot a&\mbox{ and }x\ast_{11}p=1\ast_{11}a \\
      y\cdot q=1\cdot a&\mbox{ and }y\ast_{11}q=1\ast_{11}a.
  \end{align*}
  Then $(xy)\cdot(pqa)=1\cdot a$ and
  \begin{align*}
      (xy)\ast_{11}(pqa)&=x\ast_{11}y\ast_{11}p\ast_{11}q\ast_{11}a
      =[x\ast_{11}p]\ast_{11}a\ast_{11}[y\ast_{11}q] \\
      &=[1\ast_{11}a]\ast_{11}a\ast_{11}[1\ast_{11}a]
      =1\ast_{11}a,
  \end{align*}
  then $xy\in1-a$, proving that $(1-a)(1-a)\subseteq(1-a)$.
\end{proof}

 \begin{defn}
  With the notations of Proposition \ref{prespechf} we have a functor $\Gamma:\mbox{Igr}_+\rightarrow\mbox{PSMF}$ defined by the following rules: for $R\in\mbox{Igr}_+$, $\Gamma(R)$ is the special hyperfield obtained in Proposition \ref{prespechf} and for $f\in\mbox{Igr}_+(R,S)$, $\Gamma(f):\Gamma(R)\rightarrow\Gamma(S)$ is the unique morphism such that the following diagram commute
  $$\xymatrix@!=4.5pc{R\ar[d]_{f_1}\ar[r]^{e_R} & \Gamma(R)\ar@{.>}[d]^{\Gamma(f)} \\ S\ar[r]_{e_S} & \Gamma(S)}$$
  In other words, for $x\in R$ we have
  $$\Gamma(f)(x)=(e_S\circ f_1\circ e^{-1}_R)(x)=e_S(f_1(e_R^{-1}(x))).$$
 \end{defn}

 \begin{teo}
 The functor $k:\mathcal{PSMF}\rightarrow\mbox{Igr}_+$ is the left adjoint of $\Gamma:\mbox{Igr}_+\rightarrow\mathcal{PSMF}$. The unity of the adjoint is the natural transformation $\phi:1_{\mathcal{PSMF}}\rightarrow\Gamma\circ k$ defined for $F\in\mathcal{PSMF}$ by $\phi_F=e_{k(F)}\circ\rho_F$.
 \end{teo}
 \begin{proof}
 We show that for all $f\in\mathcal{PSMF}(F,\Gamma(R))$ there is an unique $f^\sharp:\mbox{Igr}_+(k(F),R)$ such that $\Gamma(f^\sharp)\circ\phi_F=f$. Note that $\phi_F=e_{k(F)}\circ\rho_F$ is an isomorphism (because $e_{k(F)}$ and $\rho_F$ are isomorphisms).

 Let $f^\sharp_0:1_{\mathbb F_2}:\mathbb F_2\rightarrow\mathbb F_2$ and
 $f^\sharp_1:=e_R^{-1}\circ f\circ(\phi_F)^{-1}\circ e_{k(F)}:k_1(F)\rightarrow R_1$. For $n\ge2$, define $h_n:\prod^n_{i=1}k_1(F)\rightarrow R_n$ by the rule
 $$h_n(\rho(a_1),...,\rho(a_n)):=e_R^{-1}(f(a_1))\ast...\ast e_R^{-1}(f(a_n)).$$
 We have that $h_n$ is multilinear and by the Universal Property of tensor products we have an induced morphism $\bigotimes^n_{i=1}k_n(F)\rightarrow R_n$ defined on the generators by
 $$h_n(\rho(a_1)\otimes...\otimes\rho(a_n)):=e_R^{-1}(f(a_1))\ast...\ast e_R^{-1}(f(a_n).$$
 Now let $\eta\in Q_n(F)$. Suppose without loss of generalities that $\eta=\rho(a_1)\otimes...\otimes\rho(a_n)$ with $a_1\in1-a_2$. Then $f(a_1)\in1-f(a_2)$ which imply $e_R^{-1}(f(a_1))\in1-e_R^{-1}(f(a_2))$. Since $R_n\in\mbox{Igr}_+$,
 \begin{align*}
     h_n(\eta):=h_n(\rho(a_1)\otimes...\otimes\rho(a_n))
     =e_R^{-1}(f(a_1))\ast...\ast e_R^{-1}(f(a_n))=0\in R_n.
 \end{align*}
 Then $h_n$ factors through $Q_n$, and we have an induced morphism $\overline h_n:k_n(F)\rightarrow R_n$. We set $f^{\sharp}_n:=\overline h_n$. In other words, $f^\sharp_n$ is defined on the generators by
 $$f^\sharp_n(\rho(a_1)...\rho(a_n)):=e_R^{-1}(f(a_1))\ast...\ast e_R^{-1}(f(a_n).$$
 Finally, we have
 \begin{align*}
   \Gamma(f^\sharp)\circ\phi_F
   &=[e_{R}\circ(f^\sharp_1)\circ e^{-1}_{k(F)}]\circ[e_{k(F)}\circ\rho_F]
   =e_{R}\circ(f^\sharp_1)\circ\rho_F \\
   &=e_{R}\circ[e_R^{-1}\circ f\circ(\phi_F)^{-1}\circ e_{k(F)}]\circ\rho_F \\
   &=f\circ(\phi_F)^{-1}\circ[e_{k(F)}\circ\rho_F] \\
   &=f\circ(\phi_F)^{-1}\circ\phi_F=f.
 \end{align*}

 For the uniqueness, let $u,v\in\mbox{Igr}_+(k(F),R)$ such that $\Gamma(u)\circ\phi_F=\Gamma(v)\circ\phi_F$. Since $\phi_F$ is an isomorphism we have $u_1=v_1$ and since $k(F)\in\mbox{Igr}_+$ we have $u=v$.
 \end{proof}

\section{Final Remarks and future works}

In \cite{roberto2021quadratic} we emphasize that DM-multirings and DP-multirings provide a new way to look at the abstract theories of quadratic forms (for example, for special groups we obtained an easy way to describe the axiom SG6 in the theory of special groups).

In \cite{ribeiro2021vonNeumannHull} is constructed a von Neumann hull functor from multiring category and that, when restricted to semi-real rings, it commutes with real semigroup functor. This allows us to obtain some quadratic forms properties of a semi-real ring by looking to its von Neumann regular hull. In fact, the Witt ring of a real semigroups is canonically isomorphic to the Witt ring of its Von Neumann regular hull and this connection allows the full description of the graded Witt ring of a real semigroup.

With these two observations in mind, it is natural consider possible expansions of ``Milnor's triangle'' of graded rings -- where K-theory is interpolating graded Witt ring and graded Cohomology ring -- to the multiring setting. In more details, below are some further roads to follow:

\begin{enumerate}
\item Extension of the K-theory framework to more general multirings (for example, to VN-multirings) with quadratic flavour.

\item Compare graded K-theory with graded Witt ring for VN-real semigroups as in the field case (Milnor \cite{milnor1970algebraick}) and special groups (Dickmann-Miraglia \cite{dickmann2000special}).

\item The definition and analysis of the  structure of Witt ring (and graded Witt ring) of  more general quadratic structures (not only obtained from special groups and real semigroups): this subject have already  appeared in section 4, in connection with \cite{worytkiewiczwitt2020witt}.

\item Still in the hyperfield case, investigate the extension of the concept of Galois group to hyperfields, comparing the Galois cohomology ring and analyse the existence of some canonical arrow from K-theory to this cohomology ring, in an attempt to recover the Milnor's Conjecture available in the classic algebraic quadratic forms context (\cite{roberto2022galoisSG}, \cite{roberto2022graded}).
\end{enumerate}

\section*{Acknowledgement}
We want to thank the anonymous referee for her/his careful reading and valuable suggestions.

\bibliographystyle{amsplain}
\bibliography{one_for_all}
\end{document}